\documentclass{amsart}

\usepackage{amsmath,amssymb}
\usepackage{mathrsfs}
\usepackage{mathtools}
\usepackage{stmaryrd}
\usepackage{enumerate}
\usepackage{tikz-cd}
\usepackage[all]{xy}
\usepackage{aliascnt}
\usepackage[colorlinks=true, linkcolor=blue, citecolor=blue]{hyperref}

\usepackage{enumitem}

\usepackage{combelow}

\newtheorem{maintheorem}{Theorem}

\newaliascnt{maincorollary}{maintheorem}

\aliascntresetthe{maincorollary}

\newtheorem{theorem}{Theorem}

\newaliascnt{lemma}{theorem}
\newtheorem{lemma}[lemma]{Lemma}
\aliascntresetthe{lemma}

\newaliascnt{corollary}{theorem}
\newtheorem{corollary}[corollary]{Corollary}
\aliascntresetthe{corollary}

\newaliascnt{proposition}{theorem}
\newtheorem{proposition}[proposition]{Proposition}
\aliascntresetthe{proposition}

\newaliascnt{conjecture}{theorem}

\aliascntresetthe{conjecture}

\newaliascnt{question}{theorem}

\aliascntresetthe{question}

\theoremstyle{definition}

\newaliascnt{definition}{theorem}
\newtheorem{definition}[definition]{Definition}
\aliascntresetthe{definition}

\newaliascnt{remark}{theorem}
\newtheorem{remark}[remark]{Remark}
\aliascntresetthe{remark}

\newaliascnt{example}{theorem}
\newtheorem{example}[example]{Example}
\aliascntresetthe{example}

\newaliascnt{notation}{theorem}
\newtheorem{notation}[notation]{Notation}
\aliascntresetthe{notation}

\newif\ifhascomments \hascommentstrue
\ifhascomments
  \newcommand{\matt}[1]{{\color{red}[[\ensuremath{\spadesuit\spadesuit\spadesuit} #1]]}}
  \newcommand{\jeremy}[1]{{\color{red}[[\ensuremath{\clubsuit\clubsuit\clubsuit} #1]]}}
\else
  \newcommand{\matt}[1]{}
  \newcommand{\jeremy}[1]{}
\fi

\renewcommand{\setminus}{\smallsetminus}

\newcommand{\Z}{\mathbb{Z}}
\newcommand{\bZ}{\mathbb{Z}}
\newcommand{\ZZ}{\mathbb{Z}}

\newcommand{\QQ}{\mathbb{Q}}
\newcommand{\Q}{\mathbb{Q}}
\newcommand{\bQ}{\mathbb{Q}}
\newcommand{\C}{\mathbb{C}}
\newcommand{\R}{\mathbb{R}}
\newcommand{\bF}{\mathbb{F}}

\newcommand{\g}{\mathfrak{g}}

\newcommand{\cT}{\mathscr{T}}
\newcommand{\cX}{\mathcal{X}}
\newcommand{\cY}{\mathcal{Y}}
\newcommand{\cZ}{\mathcal{Z}}

\newcommand{\cI}{\mathcal{I}}

\newcommand{\cC}{\mathcal{C}}
\newcommand{\cD}{\mathcal{D}}
\newcommand{\cU}{\mathcal{U}}
\newcommand{\cO}{\mathcal{O}}

\newcommand{\cG}{\mathcal{G}}
\newcommand{\cH}{\mathcal{H}}
\newcommand{\cJ}{\mathscr{J}}
\newcommand{\mcJ}{\mathcal{J}}

\newcommand{\cE}{\mathcal{E}}

\newcommand{\cL}{\mathcal{L}}
\newcommand{\cM}{\mathcal{M}}

\newcommand{\sL}{\mathscr{L}}
\newcommand{\sM}{\mathscr{M}}

\newcommand{\sJ}{\mathscr{J}}

\newcommand{\sW}{\mathscr{W}}

\newcommand{\bL}{\mathbb{L}}

\newcommand{\bG}{\mathbb{G}}
\newcommand{\bA}{\mathbb{A}}

\newcommand{\diff}{\mathrm{d}}
\newcommand{\red}{\mathrm{red}}
\newcommand{\id}{\mathrm{id}}
\newcommand{\Gor}{\mathrm{Gor}}

\newcommand{\str}{\mathrm{str}}

\newcommand{\synt}{\mathrm{synt}}
\newcommand{\inj}{\mathrm{inj}}

\DeclareMathOperator{\rep}{rep}

\DeclareMathOperator{\EP}{EP}

\DeclareMathOperator{\age}{age}

\DeclareMathOperator{\Res}{Res}

\DeclareMathOperator{\het}{ht}

\DeclareMathOperator{\e}{e}
\DeclareMathOperator{\Spec}{Spec}
\DeclareMathOperator{\Ext}{Ext}
\DeclareMathOperator{\ord}{ord}
\DeclareMathOperator{\Hom}{Hom}
\DeclareMathOperator{\Isom}{Isom}
\DeclareMathOperator{\Aut}{Aut}
\DeclareMathOperator{\GL}{GL}

\DeclareMathOperator{\coh}{coh}
\DeclareMathOperator{\sm}{sm}

\DeclareMathOperator{\coker}{coker}

\DeclareMathOperator{\SL}{SL}

\DeclareMathOperator{\uHom}{\underline{\Hom}}
\DeclareMathOperator{\uIsom}{\underline{\Isom}}
\DeclareMathOperator{\uAut}{\underline{\Aut}}

\DeclareMathOperator{\wt}{wt}
\DeclareMathOperator{\shft}{shft}

\DeclareMathOperator{\Gal}{Gal}
\DeclareMathOperator{\cha}{char}
\DeclareMathOperator{\Ann}{Ann}
\DeclareMathOperator{\Conj}{Conj}

\DeclareMathOperator{\Tr}{Tr}
\DeclareMathOperator{\Supp}{Supp}
\DeclareMathOperator{\Irrep}{Irrep}
\DeclareMathOperator{\gr}{gr}
\DeclareMathOperator{\codim}{codim}
\DeclareMathOperator{\Proj}{Proj}
\DeclareMathOperator{\colim}{colim}
\DeclareMathOperator{\Stab}{Stab}

\newcommand{\grp}{\mathrm{grp}}

\tikzset{cong/.style={draw=none,edge node={node [sloped, allow upside down, auto=false]{$\cong$}}},
         Isom/.style={above,every to/.append style={edge node={node [sloped, allow upside down, auto=false]{$\sim$}}}}}

\title[McKay correspondence in positive characteristic]{McKay correspondence for linearly reductive finite group schemes in positive characteristic}

\author{Matthew Satriano and Jeremy Usatine}

\thanks{MS was partially supported by an NSERC Discovery Grant. JU was partially supported by Simons Foundation MPS-TSM-00007918 and NSF DMS-2502347.}

\address{Matthew Satriano, Department of Pure Mathematics, University of Waterloo}
\email{msatriano@uwaterloo.ca}

\address{Jeremy Usatine, Department of Mathematics, Florida State University}
\email{jusatine@fsu.edu}

\begin{document}

\begin{abstract}
We obtain a motivic and a cohomological McKay correspondence for finite linearly reductive group schemes in arbitrary characteristic. In particular, we prove that if $V$ is a finite dimensional vector space and $G$ is a finite linearly reductive subgroup scheme of $\mathrm{SL}(V)$, then the Euler number of any crepant resolution of $V/G$ is equal to the number of irreducible algebraic representations of $G$. We obtain these McKay correspondences as a consequence of a motivic change of variables formula applied to $[V/G] \to V/G$. If $G$ is non-reduced, as can happen in positive characteristic, the stack quotient $[V/G]$ is not Deligne-Mumford. Therefore in order to prove this change of variables formula and the resulting McKay correspondences, we generalize the authors' theory of motivic integration for Artin stacks to arbitrary characteristic, which may be of independent interest.
\end{abstract}

\maketitle

\setcounter{tocdepth}{1}

\tableofcontents

\section{Introduction}

\numberwithin{theorem}{section}
\numberwithin{lemma}{section}
\numberwithin{corollary}{section}
\numberwithin{proposition}{section}
\numberwithin{conjecture}{section}
\numberwithin{question}{section}
\numberwithin{remark}{section}
\numberwithin{definition}{section}
\numberwithin{example}{section}
\numberwithin{notation}{section}

A version of the McKay correspondence (henceforth referred to as the Euler number McKay correspondence), conjectured by Reid \cite{Reid} and proved by Batyrev \cite{BatyrevNonArch}, states that if $G$ is a finite subgroup of $\SL_n(\C)$, then the Euler number of any crepant resolution of $\C^n / G$ is equal to the number of irreducible $\C$-representations of $G$ (or equivalently, the number of conjugacy classes of $G$). This beautiful theorem is a vast (partial) generalization of the ``classical McKay correspondence'', which considers the special case where $n = 2$, and has inspired many interesting avenues since then (far more than we could adequately survey here). In particular, Batyrev proved a cohomological refinement of the Euler number McKay correspondence \cite{BatyrevNonArch} (conjectured in \cite{BatyrevDais}, see also \cite{ItoReid}) and Denef and Loeser, as one of the early seminal applications of motivic integration, proved a motivic refinement of the Euler number McKay correspondence \cite{DenefLoeser}.

The Euler number McKay correspondence (and its motivic refinement) works just as well if $\C$ is replaced with an arbitrary field $k$, as long as one additionally assumes that $\#G$ is prime to the characteristic of $k$, i.e., $G$ is a tame finite group. But if the finite group $G$ is not tame, the situation becomes far more complicated. This is the subject of Yasuda's recent ``wild motivic McKay correspondence'' \cite{Yasuda2024}, which for example, implies a motivic version of Bhargava's mass formula \cite{Bhargava}. Needless to say, the McKay correspondence outside the case of tame finite groups is a deep subject with many aspects that have yet to be explored (see \cite{Yasuda2023} for a survey of some recent developments and remaining questions). In that vein, another avenue is to consider the McKay correspondence for $G$ that are finite group \emph{schemes}, e.g., as called for in \cite[Problem 10.2]{Yasuda2023}. In positive characteristic, even over algebraically closed fields, finite group schemes need not be constant groups: they can be non-reduced, such as $\alpha_p$ and $\mu_p$ in characteristic $p$. 

In this paper we obtain an Euler number McKay correspondence, as well as motivic and cohomological refinements, for finite group schemes that are \emph{linearly reductive}. In particular, we show that if $V$ is a finite dimensional vector space over a field $k$ and $G$ is a finite linearly reductive subgroup scheme of $\SL(V)$, then the Euler number of any crepant resolution of $V/G$ is equal to the number of irreducible algebraic $\overline{k}$-representations of $G \otimes_k \overline{k}$, where $\overline{k}$ is an algebraic closure of $k$.

Over $\C$, finite linearly reductive group schemes are simply constant group schemes, i.e., finite groups. However, in characteristic $p$, they have a richer structure. They include group schemes that are connected and non-reduced such as $\mu_p$. Although this non-reduced structure may be unsightly at first, such group schemes are fundamental in positive characteristic. They are the key objects used in the theory of tame stacks developed by Abramovich--Olsson--Vistoli \cite{AOV}. Additionally, the first author proved that the classical Chevalley--Shephard--Todd Theorem \cite{SatrianoCST} holds for such group schemes, and Liedtke--Martin--Matsumoto \cite{LiedtkeMartinMatsumoto} showed that quotients by such group schemes retain many of the central features of classical quotient singularities.

Let us illustrate with a concrete example one reason why finite linearly reductive group schemes are necessary in the study of singularities in positive characteristic. Consider the $A_1$-singularity given by $xy=z^2$. For characteristics other than $2$, this variety is obtained as the quotient $\bA^2/(\bZ/2)$ where $\bZ/2$ acts as $(x,y)\mapsto(-x,-y)$. In characteristic $2$, however, this action is trivial so this quotient does not recover the $A_1$-singularity. In fact, one can prove that the $A_1$-singularity is not the quotient $\bA^2/G$ for any finite group $G$. On the other hand, the non-reduced group scheme $\mu_2$ acts on $\bA^2$ with weights $(1,1)$, and a direct calculation shows that $\bA^2/\mu_2$ yields the $A_1$-singularity. In a similar vein, many rational double point surface singularities, which are particularly relevant to the McKay correspondence in dimension 2, cannot be obtained in positive characteristic as quotients of $\bA^2$ by finite groups. Yet all rational double point surface singularities arise in positive characteristic (when the characteristic is at least 7) as quotients of $\bA^2$ by finite linearly reductive group schemes \cite{Hashimoto, LiedtkeSatrianoBirational, LiedtkeSatriano25}.

To our knowledge, the first paper to consider the McKay correspondence for finite linearly reductive group schemes in positive characteristic is \cite{Liedtke}, in which Liedtke obtains a classical McKay correspondence, i.e., the special case where the dimension is $n=2$. See \autoref{subsectionSurfaces} below for some discussion on how our results, specialized to dimension 2, compare to \cite{Liedtke}.

\subsection{Main results}

If $k$ is a field, we will let $\widehat{\sM}_k$ denote the ring obtained by inverting the class of $\bA^1_k$ in the Grothendieck ring of $k$-varieties and then completing with respect to the dimension filtration. For any finite type Artin stack $\cY$ over $k$ with affine geometric stabilizers, we will let $\e(\cY) \in \widehat{\sM}_k$ denote its class in $\widehat{\sM}_k$ (see e.g., \cite[Section 2.2]{SatrianoUsatine1}), and we will set $\bL = \e(\bA^1_k)$. If $Y$ is a $\Q$-Gorenstein variety over $k$, Denef and Loeser introduced \cite{DenefLoeser} a certain function on the arc scheme $\sL(Y)$ of $Y$ that we will denote by $\ord^\Gor_Y: \sL(Y) \to \Q \cup \{\infty\}$ (see \autoref{subsectionGorensteinMeasure} below for details). When $\bL^{\ord^\Gor_Y}$ is integrable on $\sL(Y)$ in the sense of motivic integration, they used it to define a class
\[
	\e_{\str}(Y) = \bL^{\dim Y} \int_{\sL(Y)} \bL^{\ord^\Gor_Y} \diff\mu_Y \in \widehat{\sM}_k[\bL^{1/m}],
\]
where $m$ is such that $Y$ is $m$-Gorenstein. This construction has some important properties that we recall.

\begin{itemize}

\item If $Y$ admits a crepant resolution of singularities $X \to Y$ by a scheme $X$, then $\bL^{\ord^\Gor_Y}$ is integrable on $\sL(Y)$, and $\e_{\str}(Y) = \e(X)$.

\item When $k$ has characteristic 0, $\bL^{\ord^\Gor_Y}$ is integrable on $\sL(Y)$ if and only if $Y$ has log-terminal singularities. In that case, $\e_{\str}(Y)$ specializes to the stringy $E$-function introduced by Batyrev in \cite{Batyrev}.

\end{itemize}

Our first main result is a motivic McKay correspondence that gives a representation theoretic description for $\e_{\str}(V/G)$ when $V$ is a finite dimensional vector space over $k$ and $G$ is a finite linearly reductive subgroup scheme of $\SL(V)$. In characteristic 0, such a description has traditionally been written as a sum over conjugacy classes of $G$. In our setting, we need the following adaptation, which is essentially to replace points of $G$ with maps from $\mu_r$. While this use of maps from $\mu_r$ (up to conjugacy) falls out directly from the methods of this paper, we also note that in characteristic 0, the idea to consider the McKay correspondence in terms of maps from $\mu_r$ (up to conjugacy) goes back at least as far \cite{ItoReid}.

\begin{notation}
Let $G$ be a finite linearly reductive group scheme over an algebraically closed field $k$. For any $r \in \Z_{>0}$, there exists a finite scheme $\uHom_k^{\grp, \inj}(\mu_r, G)$ over $k$ parametrizing group scheme monomorphisms from $\mu_r$ to $G$ by \cite[Proposition 1.1]{Sala}. The conjugation action of $G$ on itself induces an action of $G$ on $\uHom_k^{\grp, \inj}(\mu_r, G)$, and we set
\[
	\Conj_{\mu_r}(G) = (\uHom_k^{\grp, \inj}(\mu_r, G) / G)(k)
\]
and
\[
	\Conj_\mu(G) = \bigsqcup_{r \in \Z_{>0}} \Conj_{\mu_r}(G).
\]
If $\phi \in \Conj_\mu(G)$, we will let $Z_G(\phi)$ denote its centralizer in $G$, i.e., $Z_G(\phi)$ is the centralizer of $\widetilde{\phi}: \mu_r \to G$ for any choice of $\widetilde{\phi}$ whose image in $\Conj_\mu(G)$ is $\phi$. Note that as a group scheme up to isomorphism, $Z_G(\phi)$ does not depend on the choice of $\widetilde{\phi}$.
\end{notation}

We will also need the notion of \emph{age}, as we now define. This is just as defined in \cite[Proof of Theorem 1.3]{ItoReid}, except we have written the definition to also make sense when $\mu_r$ is not reduced, as can be the case in positive characteristic.

\begin{definition}
Let $V$ be a finite dimensional vector space over an algebraically closed field $k$, let $G$ be a finite linearly reductive subgroup scheme of $\GL(V)$, and let $\phi \in \Conj_\mu(G)$. We will define the \emph{age} of $\phi$ as follows. Choose some $\widetilde{\phi}:  \mu_r \to G$ whose image in $\Conj_\mu(G)$ is $\phi$. Then the restriction along $\widetilde{\phi}$ of the $G$-action on $V$ gives a linear $\mu_r$-action on $V$. Equivalently, we have a $\Z/r$-grading $V = \bigoplus_{w = 1}^{r} V_w$. Then we set
\[
	\age(\phi) = (1/r)\sum_{w = 1}^r w \dim_k V_w,
\]
and we note this does not depend on the choice of $\widetilde{\phi}$. Furthermore, if $G$ is a subgroup scheme of $\SL(V)$, then $\age(\phi)$ is an integer by \autoref{corollaryAgeIsInteger} below.
\end{definition}

We are now prepared to state the first main result of this paper.

\begin{maintheorem}[Motivic McKay correspondence]\label{maintheoremMotivicMcKay}
Let $V$ be a finite dimensional vector space over a field $k$, and let $G$ be a finite linearly reductive subgroup scheme of $\SL(V)$. Then $V/G$ is Gorenstein, and $\bL^{\ord_{V/G}^\Gor}$ is integrable on $\sL(V/G)$. Furthermore if $k$ is algebraically closed, then
\[
	\e_{\str}(V/G) = \sum_{\phi \in \Conj_{\mu}(G)} \bL^{\age(\phi)} \e(B Z_G(\phi) ).
\]
\end{maintheorem}

\begin{remark}
Consider the special case where $k$ has characteristic 0. While the equality in \autoref{maintheoremMotivicMcKay} occurs in $\widehat{\sM}_k$, the motivic McKay correspondence of Denef and Loeser \cite[Theorem 3.6]{DenefLoeser} occurs in a modified ring $\widehat{\sM}_/$ (see 1.10 in loc. cit.). By \cite[Proposition 3.1(ii)]{Ekedahl2}, the canonical map $\widehat{\sM}_k \to \widehat{\sM}_/$ sends each $\e(B Z_G(\phi) )$ to $1$. On the other hand by \cite[Corollary 5.2]{Ekedahl}, there are finite groups whose classifying space have nontrivial class in $\widehat{\sM}_k$.
\end{remark}

The next main result is an Euler number McKay correspondence in our setting. In the following, $\chi_{\str}(V/G)$ denotes Batyrev's \emph{stringy Euler number} of $V/G$, which is obtained by applying the Euler-Poincar\'{e} specialization $\EP: \widehat{\sM}_k \to \Z\llparenthesis t^{-1} \rrparenthesis$ to $\e_{\str}(V/G)$ and then taking the limit\footnote{In characteristic 0, this limit always exists. This limit also exists under the hypotheses of \autoref{maincorollaryEulerMcKay} by \autoref{maintheoremMotivicMcKay}.} as $t \to 1$. In particular, if $V/G$ admits a crepant resolution of singularities $X \to V/G$ by a scheme $X$, then $\chi_{\str}(V/G)$ equals the (compactly supported) Euler number of $X$.

\begin{maintheorem}[Euler number McKay correspondence]\label{maincorollaryEulerMcKay}
Let $V$ be a finite dimensional vector space over a field $k$, let $\overline{k}$ be an algebraic closure of $k$, and let $G$ be a finite linearly reductive subgroup scheme of $\SL(V)$. Then
\[
	\chi_{\str}(V/G) = \#\Irrep_{\overline{k}}(G \otimes_k \overline{k}) =  \#\Conj_{\mu}(G \otimes_k \overline{k}),
\]
where $\Irrep_{\overline{k}}(G \otimes_k \overline{k})$ denotes the set of irreducible algebraic $\overline{k}$-representations (up to isomorphism) of $G \otimes_k \overline{k}$.
\end{maintheorem}

\begin{remark}
Although each $\age(\phi)$, and thus the expression for $\e_{\str}(V/G)$ in \autoref{maintheoremMotivicMcKay}, depends on the embedding $G \hookrightarrow \SL(V)$, the sets $\Conj_\mu(G \otimes_k \overline{k})$ and $\Irrep_{\overline{k}}(G \otimes_k \overline{k})$ only depend on the algebraic group $G$. Therefore \autoref{maincorollaryEulerMcKay} shows that $\chi_\str(V/G)$ depends only on $G$ as a group scheme and not on how it is embedded in $\SL(V)$.
\end{remark}

We also obtain the following cohomological version of the McKay correspondence in our setting.

\begin{maintheorem}[Cohomological McKay correspondence]\label{maintheoremCohomologicalMcKay}
Let $k$ be a field, let $k^s$ be a separable closure of $k$, let $\overline{k}$ be an algebraic closure of $k$, and let $\ell$ be a prime number invertible in $k$. Let $V$ be a finite dimensional vector space over $k$, and let $G$ be a finite linearly reductive subgroup scheme of $\SL(V)$. If $X \to V/G$ is a crepant resolution of singularities by a scheme $X$, then for all $i \in \Z$,
\[
	\dim_{\Q_\ell} H_{\mathrm{\acute{e}t,c}}^i(X \otimes_k k^s, \Q_\ell) = \#\{\phi \in \Conj_{\mu}(G \otimes_k \overline{k}) \, | \, \age(\phi) = i/2\}.
\]
\end{maintheorem}

\begin{remark}
Since each $\age(\phi)$ is an integer, \autoref{maintheoremCohomologicalMcKay} implies that $X$ only has nonvanishing (compactly supported) cohomology in even degree.
\end{remark}

\subsection{Motivic integration for Artin stacks}

A key result of this paper that is used to prove our McKay correspondences is the following formula for $\e_{\str}(Y)$ in terms of a crepant resolution by a tame Artin stack. It can be thought of as a global version of the McKay correspondence, and in the special case of tame \emph{Deligne--Mumford} stacks (for example, in characteristic 0), versions of it have been proved in \cite{LupercioPoddar, Yasuda2004, Yasuda}. Our main results involve applying this global McKay correspondence to the stack $[V/G]$. We emphasize that when $G$ is non-reduced as can happen in positive characteristic, $[V/G]$ is a tame Artin stack that is \emph{not} Deligne--Mumford, so the previously known results in \cite{LupercioPoddar, Yasuda2004, Yasuda} are not sufficient for the applications in this paper.  

In the following, $K_{\cX/Y}$ denotes the relative canonical divisor of $\cX$ over $Y$ as defined in \cite[Definition 1.6]{SatrianoUsatine3}. See \autoref{subsectionWeightAndShift} below for the definitions of $\shft_\cX(\cY)$ and $I_{\mu}(\cX)$.

\begin{maintheorem}\label{maintheoremGorensteinMeasureCrepantResolution}
Let $k$ be a field, let $\cX$ be a smooth finite type irreducible Artin stack over $k$, and let $Y$ be a $\Q$-Gorenstein finite type irreducible scheme over $k$. If $\cX \to Y$ is a tame proper birational map with $K_{\cX/Y} = 0$, then $\bL^{\ord_Y^\Gor}$ is integrable on $\sL(Y)$, and
\[
	\e_{\str}(Y) = \sum_{\cY} \bL^{\shft_\cX(\cY)} \e(\cY),
\]
where the sum varies over all connected components $\cY$ of $I_\mu(\cX)$. Furthermore if $Y$ is $1$-Gorenstein, then $\shft_\cX(\cY) \in \Z$ for all connected components $\cY$ of $I_\mu(\cX)$.
\end{maintheorem}

\begin{remark}
Each $\cY$ is a finite type Artin stack over $k$ with affine geometric stabilizers by \autoref{propositionCyclotomicInertiaFiniteTypeAffineDiagonalSmooth} below. In particular, the classes $\e(\cY)$ in the statement of \autoref{maintheoremGorensteinMeasureCrepantResolution} are well defined.
\end{remark}

We obtain \autoref{maintheoremGorensteinMeasureCrepantResolution} as a consequence of the authors' theory of motivic integration for Artin stacks, as developed in \cite{SatrianoUsatine1, SatrianoUsatine2, SatrianoUsatine3, SatrianoUsatine4, SatrianoUsatine5, SatrianoUsatine6}. Much of the effort, which is necessary for the main results above and may also be of independent interest, is to extend this theory to the setting of positive characteristic. We note that since tame Artin stacks in many ways behave much better than Artin stacks in general (for example, they admit a coarse moduli space, and the canonical map to their coarse moduli space is separated), one might expect that the full theory of motivic integration for Artin stacks would not be necessary to obtain the main results of this paper. Along those lines, one might expect that the theory of motivic integration for (tame) Deligne--Mumford stacks (such as developed in \cite{Yasuda}) could be pushed to the needed level of generality with less effort. However, we believe this is unlikely to be possible for the following reasons. The essential reason why motivic integration for Artin stacks requires additional techniques beyond the Deligne--Mumford case is that unlike smooth Deligne--Mumford stacks, smooth Artin stacks may have a cotangent complex that is not concentrated in degree 0. There are many relevant downstream consequences of this. For example in the (smooth) Deligne--Mumford case, the jet stack truncation morphisms have fibers that are isomorphic to affine spaces, while in the (smooth) Artin case, these fibers vary and are not all even schemes. Because of this, even defining the motivic measure for untwisted arcs of Artin stacks requires additional care \cite{SatrianoUsatine1} not needed in the Deligne--Mumford case. Similar downstream consequences affect much of the theory of motivic integration for Artin stacks, especially the authors' proofs of various change of variables formulas. Smooth tame Artin stacks, like their non-tame brethren, may have a cotangent complex that is not concentrated in degree 0. In the setting of the McKay correspondence for finite linearly reductive group schemes, this corresponds to the fact that non-reduced group schemes like $\mu_p$ have a positive dimensional Lie algebra. For this reason, even if one only cares about the McKay correspondence for finite linearly reductive group schemes in positive characteristic, we see no way to avoid the full force of techniques needed in developing motivic integration for Artin stacks.

\subsection{The special case of surfaces}\label{subsectionSurfaces} In \cite{Liedtke}, Liedtke considers the case where $\dim V = 2$ and proves a classical McKay correspondence. In particular, using very different techniques than in this paper, Liedtke proves \cite[Theorem 1.4]{Liedtke}, which implies the first equality in \autoref{maincorollaryEulerMcKay} in the special case where $\dim V = 2$. Also, Liedtke proposes that the notion of conjugacy classes for finite group schemes is an interesting avenue in its own right (see \cite[Section 1.11]{Liedtke}) and explores various versions in loc. cit. We believe the results of this paper demonstrate that $\Conj_{\mu}(G)$ is one good such notion.

\subsection*{Acknowledgments.} We thank Nathan Ilten, Elana Kalashnikov, Christian Liedtke, Sid Mathur, David McKinnon, Rahim Moosa, Martin Olsson, and Mihnea Popa for helpful discussions. We especially wish to thank Dori Bejleri, Jason Bell, Michel Brion, Patrick Brosnan, Brian Conrad, Karl Schwede, and Yash Singh for enlightening conversations about lifting group actions to resolutions of singularities.

\section{Preliminaries}

Throughout this paper, let $k$ be a field. For any stack $\cY$ over $k$, we will let $|\cY|$ denote the associated topological space. If $\cC \subset |\cY|$ and $k'$ is a field extension of $k$, we will let $\cC(k')$ denote the category of $k'$-points of $\cY$ whose equivalence class is in $\cC$, and we will let $\overline{\cC}(k')$ denote the set of isomorphism classes of $\cC(k')$. If $\cC$ is a quasi-compact locally constructible subset of $|\cY|$ for some locally finite type Artin stack $\cY$ over $k$ with affine geometric stabilizers, we will let $\e(\cC)$ denote its class in $\widehat{\sM}_k$ (see e.g., \cite[Section 2.2]{SatrianoUsatine1} for the case where $\cY$ is finite type, and e.g., \cite[Remark 2.5]{SatrianoUsatine5} for the generalization to the case where $\cY$ is only locally finite type).

\subsection{The Gorenstein measure of Denef and Loeser}\label{subsectionGorensteinMeasure} If $Y$ is an $m$-Gorenstein finite type irreducible scheme over $k$, we will let $\omega_{Y,m}$ denote the $m$-th canonical sheaf of $Y$, we will let $\sJ_{Y,m}$ denote the unique ideal sheaf on $Y$ such that the image of $(\Omega^{\dim Y}_Y)^{\otimes m} \to \omega_{Y,m}$ is $\sJ_{Y, m} \omega_{Y,m}$, and we will set
\[
	\ord^\Gor_Y = (1/m) \ord_{\sJ_{Y,m}}: \sL(Y) \to \Q \cup \{\infty\}.
\]
If $A \subset \sL(Y)$ is such that $\bL^{\ord^\Gor_Y}$ is integrable on $A$, then the \emph{Gorenstein measure} of $A$ is defined as
\[
	\mu_Y^\Gor(A) = \int_A \ord^\Gor_Y \diff\mu_Y,
\]
where $\mu_Y$ is the motivic measure on $\sL(Y)$. In particular when $\bL^{\ord^\Gor_Y}$ is integrable on $\sL(Y)$,
\[
	\e_{\str}(Y) = \bL^{\dim Y} \mu_Y^\Gor(\sL(Y)).
\]

\subsection{Weight and shift functions}\label{subsectionWeightAndShift}

The following weight function was introduced in \cite[Definition 4.17]{SatrianoUsatine6}.

\begin{notation}
Let $\cX$ be a finite type Artin stack over $k$. Assume that $\cX$ admits a cover by open substacks with affine diagonal (for example, if $\cX$ is tame and proper over a scheme), and let $r \in \Z_{> 0}$. We will let $I_{\mu_r}(\cX)$ denote the \emph{cyclotomic inertia stack} of $\cX$ of order $r$, i.e., $I_{\mu_r}(\cX)$ is the Hom stack parametrizing representable $k$-morphisms from $B\mu_r$ to $\cX$. We also set $I_\mu(\cX) = \bigsqcup_{r \in \Z_{>0}} I_{\mu_r}(\cX)$.
\end{notation}

\begin{proposition}\label{propositionCyclotomicInertiaFiniteTypeAffineDiagonalSmooth}
Let $\cX$ be a finite type Artin stack over $k$. Assume that $\cX$ admits a cover by open substacks with affine diagonal (for example, if $\cX$ is tame and proper over a scheme), and let $r \in \Z_{> 0}$. Then $I_{\mu_r}(\cX)$ is a finite type Artin stack over $k$ with affine geometric stabilizers. Furthermore if $\cX$ is smooth over $k$, then $I_{\mu_r}(\cX)$ is smooth over $k$.
\end{proposition}

\begin{proof}
By \cite[Proposition 4.16]{SatrianoUsatine6}, whose proof never uses that $k$ is algebraically closed or characteristic 0, $I_{\mu_r}(\cX)$ is a finite type Artin stack over $k$ with affine geometric stabilizers. The same argument, with \cite[Theorem 3.12(v)]{Rydh} in place of \cite[Theorem 3.12(xv, xxi)]{Rydh}, gives the last statement.
\end{proof}

\begin{definition}\label{definitionCyclotomicInertiaAndWeightFunction}
Let $\cX$ be a finite type Artin stack over $k$. Assume that $\cX$ admits a cover by open substacks with affine diagonal (for example, if $\cX$ is tame and proper over a scheme), and let $r \in \Z_{> 0}$. If $\varphi: B\mu_{r, k'} \to \cX$ is a $k'$-point of $I_{\mu_r}(\cX)$, we set
\begin{align*}
	&\overline{\wt}_\cX(\varphi) =\\
	&\dim\cX + (1/r)\sum_{w = 1}^r w[ \dim_{k'} H^1((L\varphi^* L_\cX(-w)) - \dim_{k'}H^0((L\varphi^*L_\cX)(-w))].
\end{align*}
We get an induced \emph{weight function}
\[
	\overline{\wt}_\cX: |I_\mu(\cX)| \to \Q.
\]
\end{definition}

A key feature of $\overline{\wt}_\cX$ is the following.

\begin{proposition}\label{propositionWeightFunctionLocallyConstant}
Let $\cX$ be a smooth finite type equidimensional Artin stack over $k$. Assume that $\cX$ admits a cover by open substacks each with affine diagonal and a good moduli space (for example, if $\cX$ is tame and proper over a scheme, or if $\cX$ has affine diagonal and is tame over $k$). Then $\overline{\wt}_\cX: |I_\mu(\cX)| \to \Q$ is locally constant.
\end{proposition}

\begin{proof}
The conclusion can be checked Zariski locally, so we may assume that $\cX$ has affine diagonal and a good moduli space. Then this is \cite[Proposition 4.19]{SatrianoUsatine6}, whose proof never uses that $k$ is algebraically closed or characteristic 0.
\end{proof}

\autoref{propositionWeightFunctionLocallyConstant} allows us to set the following notation.

\begin{notation}
Let $\cX$ be a smooth finite type equidimensional Artin stack over $k$. Assume that $\cX$ admits a cover by open substacks each with affine diagonal and a good moduli space (for example, if $\cX$ is tame and proper over a scheme, or if $\cX$ has affine diagonal and is tame over $k$). For any connected component $\cY$ of $I_{\mu_r}(\cX)$, we let $\overline{\wt}_\cX(\cY)$ denote the value $\overline{\wt}_\cX$ takes on $|\cY|$, and we set
\[
	\shft_\cX(\cY) = \dim\cX - \dim\cY - \overline{\wt}_\cX(\cY).
\]
\end{notation}

\begin{remark}
When $\cX$ is tame and $k$ has characteristic 0, $\shft_\cX(\cY)$ coincides with the shift defined in \cite{Yasuda2004,Yasuda} by \cite[Remark 2.11 and Lemma 9.6]{HuangSatrianoUsatine}.
\end{remark}

\subsection{Miscellaneous} We end this section by stating some miscellaneous results that will be used in the paper.

\begin{proposition}\label{propositionDimensionAndGrothendieckNorm}
Let $\cY$ be a finite type Artin stack over $k$ with affine geometric stabilizers. Then
\[
	\Vert \e(\cY) \Vert = \exp(\dim\cY).
\]
\end{proposition}

\begin{proof}
This is \cite[Proposition 4.4]{SatrianoUsatine5} except without any assumptions on the field $k$. The proof in loc. cit. works verbatim except for the use of resolution of singularities at the end of the proof of \cite[Lemma 4.2]{SatrianoUsatine5}. However, that use of resolution of singularities was unnecessary. It was only used to justify the special case of \cite[Lemma 4.2]{SatrianoUsatine5} where $\cY$ is a scheme, and that special case is already known for arbitrary $k$, see e.g., \cite[Chapter 2 Corollary 3.5.12(a)]{ChambertLoirNicaiseSebag}.
\end{proof}

\begin{lemma}\label{l:fppf-loc-sm->sm}
Let $\cX$ be an Artin stack over a base scheme $S$. Let $X\to\cX$ be an fppf cover by an algebraic space.
\begin{enumerate}
\item If $X$ is smooth over $S$, then $\cX$ is as well.
\item If $X$ is regular, then $\cX$ is as well.
\end{enumerate}
\end{lemma}
\begin{proof}
Choosing a smooth cover $U\to\cX$ by a scheme, it is enough to show $U$ is smooth (resp.~regular). The fiber product $U\times_\cX X$ is an algebraic space; choose an \'etale cover $W\to U\times_\cX X$ by a scheme. We see $W$ is smooth (resp.~regular) and $W\to U$ is an fppf cover, so Tag 05B5 (resp.~Tag 06QL) of \cite{stacks-project} shows $U$ is smooth (resp.~regular).
\end{proof}

We need the following generalization of \cite[Proposition A.1]{FantechiMannNironi}.

\begin{proposition}\label{FMN-tame-stacks}
Let $\cX$ be a normal Artin stack, $\cY$ a Noetherian tame stack with finite diagonal, and $\iota\colon\cU\to\cX$ a dominant open immersion. If $F_1,F_2\colon\cX\to\cY$ are two maps and $\beta\colon F_1\iota\Rightarrow F_2\iota$ is a $2$-isomorphism, then there exists a unique $2$-isomorphism $\alpha\colon F_1\Rightarrow F_2$ such that $\iota^*\alpha=\beta$.
\end{proposition}
\begin{proof}
The proof is exactly the same as \cite[Proposition A.1]{FantechiMannNironi} with one change. The first paragraph of the proof of (loc.~cit.) does not apply when $G$ is merely a separated group scheme; the proof also requires smoothness of $G$ since, using their notation, to deduce $\alpha_1=\alpha_2$ we need $P_1$ to be reduced. Thus, we need to show that if $\cY\to Y$ is the coarse space map, then fppf locally on $Y$, we may write $\cY=[W/G]$ with $G$ a smooth affine group scheme. This is shown in \cite[Proposition 5.2]{SatrianoCST}; note that (loc.~cit.) is stated over a perfect field but this assumption is never used. 
%\footnote{\cite[Theorem 3.2]{AOV} technically only states that fppf locally, $\cY$ is of the form $[U/G]$ with $G$ finite flat linearly reductive and hence fpqc locally on $Y$, it is well-split; however, when $Y$ is Noetherian, the proof of \cite[Lemma 2.13] produces an fppf cover where $G$ is well-split, and then the proof of \cite[Proposition 5.2]{SatrianoCST} applies to show $\cY$ fppf locally has the desired form.} 
\end{proof}

\section{Motivic integration over untwisted arcs}\label{sectionUntwistedMotivicIntegration}

In this section, we will define motivic integration over untwisted arcs of Artin stacks and prove its main structural theorems. Specifically, we will define measurable sets of untwisted arcs and their motivic measure, we will prove the change of variables formula for untwisted arcs, and we will prove that so-called ``thin subsets'' of untwisted arcs have measure zero. In the special case where the ground field $k$ has characteristic 0, these results were proved in \cite{SatrianoUsatine1, SatrianoUsatine2, SatrianoUsatine6}\footnote{In loc. cit., $k$ was also assumed to be algebraically closed, but this hypothesis was not used in the proofs of these results.}, respectively. The main contribution of this section is extending these results to the case where $k$ has positive characteristic. While many parts of the original proofs work in positive characteristic, in which case we will cite the relevant portions of loc. cit., there are also parts where we will need new arguments, which we will detail below.

We begin by setting our notation for stacks parametrizing (untwisted) jets and arcs. If $\cX$ is an Artin stack over $k$ and $n \in \Z_{\geq 0}$, we will let $\sL_n(\cX)$ denote the Weil restriction of $\cX \otimes_k k[t]/(t^{n+1})$ along the morphism $\Spec(k[t]/(t^{n+1})) \to \Spec(k)$. 

\begin{remark}
Each $\sL_n(\cX)$ is an Artin stack over $k$ by \cite[Theorem 3.7(iii)]{Rydh}. If $\cX$ is locally finite type over $k$, then each $\sL_n(\cX)$ is locally finite type over $k$ by \cite[Proposition 3.8(xii)]{Rydh}. If $\cX$ is finite type over $k$, then each $\sL_n(\cX)$ is finite type over $k$ by \cite[Proposition 3.8(xv)]{Rydh}. If $\cX$ has affine geometric stabilizers, then each $\sL_n(\cX)$ has affine geometric stabilizers by the exact same argument as in \cite[Remark 3.4]{SatrianoUsatine1}.
\end{remark}

For $m \geq n$, the truncation morphism $k[t]/(t^{m+1}) \to k[t]/(t^{n+1})$ induces a morphism $\theta^m_n: \sL_m(\cX) \to \sL_n(\cX)$. We then let $\sL(\cX)$ denote the induced inverse limit $\varprojlim_n\sL_n(\cX)$, and we let each $\theta_n: \sL(\cX) \to \sL_n(\cX)$ denote the canonical map. Note that for any field extension $k'$ of $k$, the category $\sL(\cX)(k')$ is canonically equivalent to the category of $k$-morphisms $\Spec(k'\llbracket t \rrbracket) \to \cX$, e.g., by Artin's criteria for algebraicity. We will implicitly identify these categories for the remainder of the paper, and we will refer to the objects of these categories as \emph{(untwisted) arcs}.

\begin{remark}
Although there is no reason to believe $\sL(\cX)$ is an Artin stack in general, it is at least a stack (see, e.g., \cite[Proposition 2.1.9]{Talpo}) and we will still use the notation $|\sL(\cX)|$ to denote the associated set parametrizing field valued points of $\sL(\cX)$ up to equivalence.
\end{remark}

The following special subsets of $|\sL(\cX)|$ will eventually be used to define motivic integration over $|\sL(\cX)|$.

\begin{definition}
Let $\cX$ be a locally finite type Artin stack over $k$, and let $\cC \subset |\sL(\cX)|$. The set $\cC$ is called a \emph{cylinder} if there exists some $n \in \Z_{\geq 0}$ and locally constructible subset $\cC_n \subset |\sL_n(\cX)|$ such that $\cC = \theta_n^{-1}(\cC_n)$. The set $\cC$ is called \emph{bounded} if $\theta_0(\cC)$ is contained is a quasi-compact subset of $|\sL_0(\cX)| = |\cX|$.
\end{definition}

\begin{remark}
It is straightforward to check that $\cC$ is a bounded cylinder if and only if there exists some $n \in \Z_{\geq 0}$ and quasi-compact locally constructible subset $\cC_n \subset |\sL_n(\cX)|$ such that $\cC = \theta_n^{-1}(\cC_n)$.
\end{remark}

The following is the key structural theorem that allows us to define a motivic volume for bounded cylinders.

\begin{theorem}
Let $\cX$ be a smooth equidimensional Artin stack over $k$ with affine geometric stabilizers, and let $\cC \subset |\sL(\cX)|$ be a bounded cylinder. Then $\theta_n(\cC)$ is a quasi-compact locally constructible subset of $|\sL_n(\cX)|$ for all $n \in \Z_{\geq 0}$, and the sequence
\[
	\{\e(\theta_n(\cC))\bL^{-(n+1)\dim\cX}\}_{n \in \Z_{\geq 0}} \subset \widehat{\sM}_k
\]
stabilizes for $n$ sufficiently large.
\end{theorem}

\begin{proof}
Since $\cC$ is bounded, we may replace $\cX$ with a finite type open substack containing $\theta_0(\cC)$ and therefore immediately reduce to the case where $\cX$ is finite type. The case where $\cX$ is finite type is \cite[Theorem 3.33]{SatrianoUsatine2}, where it is stated when $k$ has characteristic 0 and is algebraically closed, but neither of these hypotheses are used in the proof.
\end{proof}

As a consequence, we may define the motivic volume of a bounded cylinder.

\begin{definition}\label{definitionMotivicVolumeUntwistedArcs}
Let $\cX$ be a smooth equidimensional Artin stack over $k$ with affine geometric stabilizers, and let $\cC \subset |\sL(\cX)|$ be a bounded cylinder. Then the \emph{motivic volume} of $\cC$ is defined as
\[
	\mu_\cX(\cC) = \lim_{n \to \infty} \e(\theta_n(\cC))\bL^{-(n+1)\dim\cX} \in \widehat{\sM}_k.
\]
\end{definition}

\subsection{A technical proposition used in the proof of the change of variables formula}

The next main goal of this section is to prove the motivic change of variables formula \cite[Theorem 1.2]{SatrianoUsatine2} in positive characteristic. All of the proof in loc. cit. works verbatim except for the proof of \cite[Proposition 6.1]{SatrianoUsatine2}, which relies on \cite[Lemma 6.3]{SatrianoUsatine2}, which uses that $k$ has characteristic 0 in order to guarantee that group schemes over $k$ are smooth. The goal of this subsection is to show that \cite[Proposition 6.1]{SatrianoUsatine2} still holds when $k$ has positive characteristic. We begin with a lemma.

\begin{lemma}\label{groupCommutativeIfFibersCommutative-charp}
Let $k'$ be a field, let $S$ be a reduced locally finitely presented $k'$-scheme, and let $f\colon G\to S$ be a flat separated locally finitely presented group scheme. If for every $s \in S$, the fiber $f_s\colon G_s\to\Spec k'(s)$ is a smooth commutative group scheme, then $G$ is commutative.
\end{lemma}
\begin{proof}
The proof is the same as in \cite[Lemma 6.3]{SatrianoUsatine2}. In (loc.~cit.), the characteristic $0$ assumption was only used to prove $f_s$ is smooth.
%Let $m\colon G\times_S G\to G$ be the multiplication map and $\tau\colon G\times_S G\to G\times_S G$ be the map $\tau(g,h)=(h,g)$. To prove $G$ is commutative, we must show $m\circ\tau=m$. By \cite[Tag 0B8G]{stacks-project}, every fiber $f_s\colon G_s\to\Spec k'(s)$ is a group scheme, hence by Cartier's Theorem (see e.g., \cite[Tag 047N]{stacks-project}), every fiber $f_s$ is smooth. Since $f$ flat and locally of finite presentation, \cite[Tag 01V8]{stacks-project} shows that $f$ is smooth; in particular, $G\times_S G\to S$ is smooth. Since $S$ is reduced, $G\times_S G$ is as well. Lastly, the maps $m,m\circ\tau\colon G\times_S G\to G$ are equal on fibers, so $m\circ\tau=m$ by \autoref{equalityMapsOnFibers}.
\end{proof}

We may now prove that \cite[Proposition 6.1]{SatrianoUsatine2} holds in arbitrary characteristic.

\begin{proposition}\label{prop:paper2Prop6.1}
Let $r \in \Z_{\geq 0}$, and let $\cY$ be a finite type Artin stack over $k$. Assume that the map $\cI_\cY \to \cY$ is separated and that for every field extension $k'$ of $k$ and every $y \in \cY(k')$, the stabilizer of $y$ is isomorphic to $\bG_{a,k'}^r$. Then there exist locally closed substacks $\cY_1, \dots, \cY_m$ of $\cY$ and finite type $k$-schemes $Y_1, \dots, Y_m$ such that $|\cY| = \bigsqcup_{i = 1}^m |\cY_i|$ and $\cY_i \cong Y_i \times_k B\bG_{a,k}^r$ for all $i = 1, \dots, m$.
\end{proposition}
\begin{proof}
The proof of \cite[Proposition 6.1]{SatrianoUsatine2} goes through largely the same with some modifications. The first two paragraphs of the proof hold in arbitrary characteristic and reduce us to showing that if $\Spec K$ and $\cY\to\Spec K$ is a gerbe, then it is the trivial gerbe. At this point in the proof of \cite[Proposition 6.1]{SatrianoUsatine2}, the characteristic $0$ assumption is used to conclude that $\cY$ is smooth over $K$, hence $\cY$ is reduced; however, we can show $\cY$ is reduced without needing to assume $K$ has characteristic $0$ as follows. Recall from the first two paragraphs of the proof of \cite[Proposition 6.1]{SatrianoUsatine2} that $\cY$ is the generic fiber of some $\cZ \to W$. Since $\cZ$ arose by stratifying, we can replace it with $\cZ_{\red}$, and we can also replace $W$ with an integral affine subscheme of $W$, say $\Spec A$. To show $\cY$ is reduced, it is enough to take a smooth cover $Z\to\cZ$ and show the generic fiber of $Z \to W$ is reduced since this generic fiber is a smooth cover of $\cY$. Replacing $Z$ with an open affine, we may assume $Z=\Spec B$ is reduced. Then the generic fiber of $\Spec B \to \Spec A$ is the spectrum of a localization of $B$, hence reduced.

%It looks like where 6.3 is used, we know G is reduced, even in char p, as long as we confirm it's ok to assume \Y is reduced. Here's the reason. \Y is a gerbe so I\Y --> \Y is flat. Also all its fibers are G_a^r so smooth, and since a flat map with smooth fibers is smooth, we get I\Y --> \Y is smooth. Then G --> tilde{Y} is a base change of I\Y --> \Y so smooth. As long as we agree \Y is reduced, then tilde{Y} is reduced, so G is reduced
Replacing the use of \cite[Lemma 6.3]{SatrianoUsatine2} with \autoref{groupCommutativeIfFibersCommutative-charp}, \cite[Tag 0CJY]{stacks-project} and the remainder of the third paragraph of the proof of \cite[Proposition 6.1]{SatrianoUsatine2} show that $\cY$ is $\cG$-gerbe for some sheaf of abelian groups $\cG$ on $\Spec K$.

Since $\cY$ is a gerbe over $K$, there exists an fppf cover $T\to\Spec K$ over which $\cY$ is trivialized. Choosing any field valued point $K'$ of $T$, we see $q\colon\Spec K'\to\Spec K$ is an fpqc cover and $\cY_{K'}\simeq B\bG_a^r$. Thus, $\cG|_{K'}$ is represented by $\bG_a^r$. Then \cite[Tag 0245]{stacks-project} tells us $\cG$ is represented by an affine group scheme $G\to\Spec K$. Furthermore, $G$ is smooth over $K$ since it is smooth fpqc locally. Thus, $\cY\to\Spec K$ is smooth. Since smooth maps have sections \'etale locally, we see there is an \'etale extension $K''/K$ where $G_{K''}\simeq\bG_a^r$. Replacing $K''$ by its Galois closure, we may assume $K''/K$ is Galois. Then Hilbert's Theorem 90, tells us $H^1(\Gal(K''/K),\GL_r)=0$, hence $G=\bG_a^r$. Then the class of our $\bG_a^r$-gerbe $\cY$ lies in $H^2_{et}(\Spec K, \bG_a^r)=H^2_{et}(\Spec K, \cO^{\oplus r})=H^2_{Zar}(\Spec K, \cO^{\oplus r})=0$. It follows that $\cY=B\bG_{a,K}^r$.
\end{proof}

\subsection{The motivic change of variables formula for untwisted arcs}

We will first need some notation for the ``height function'' that provides the correction term in the change of variables formula.

\begin{notation}
Let $\cX$ be an Artin stack over $k$. For any $E \in D^-_{\coh}(\cX)$, any $i \in \Z$, and any arc $\varphi: \Spec(k'\llbracket t \rrbracket) \to \cX$, we set
\[
	\het^{(i)}_E(\varphi) = \dim_{k'} \mathbb{H}^i(L\varphi^* E) \in \Z_{\geq 0} \cup \{\infty\}.
\]
We then have an induced map 
\[
	\het^{(i)}_E: |\sL(\cX)| \to \Z_{\geq 0} \cup \{\infty\}.
\]
Let $\cY$ be a locally finite type Artin stack over $k$, and let $\cX \to \cY$ be a morphism. If $\cX$ is smooth\footnote{The following makes sense without assuming $\cX$ is smooth, but it is useful to reserve the notation for a modified version when $\cX$ is not smooth, as in \cite[Definition 6.2]{SatrianoUsatine5}.} over $k$ and $\varphi$ is an arc of $\cX$, we set
\[
	\het_{\cX/\cY}(\varphi) = \begin{cases} \het^{(0)}_{L_{\cX/\cY}}(\varphi) - \het^{(1)}_{L_{\cX/\cY}}(\varphi), \quad& \het^{(0)}_{L_{\cX/\cY}}(\varphi), \het^{(1)}_{L_{\cX/\cY}}(\varphi) \in \Z_{\geq 0} \\ \infty, & \text{otherwise} \end{cases}
\]
so we get the map
\[
	\het_{\cX/\cY}: |\sL(\cX)| \to \Z \cup \{\infty\}.
\]
\end{notation}

We will also set notation for how to integrate certain functions over spaces of arcs.

\begin{notation}
Let $\cX$ be a smooth equidimensional Artin stack over $k$ with affine geometric stabilizers, let $\cC \subset |\sL(\cX)|$ be a bounded cylinder, and let $f: \cC \to \Z$. If $f$ takes only finitely many values and the set $f^{-1}(n) \subset |\sL(\cX)|$ is a (bounded) cylinder for all $n \in \Z$, then we set
\[
	\int_\cC \bL^{f} \diff\mu_\cX = \sum_{n \in \Z} \bL^n \mu_{\cX}(f^{-1}(n)).
\]
\end{notation}

We are now prepared to state and prove the motivic change of variables formula for untwisted arcs.

\begin{theorem}\label{theoremUntwistedChangeOfVariables}
Let $\cX$ be a smooth irreducible Artin stack over $k$ with affine geometric stabilizers and separated diagonal, let $Y$ be a finite type irreducible scheme over $k$, and let $\cX \to Y$ be a morphism. Let $\cU$ be an open substack of $\cX$ such that the composition $\cU \hookrightarrow \cX \to Y$ is an open immersion, let $\cC \subset |\sL(\cX)|$ be a bounded cylinder that is disjoint from $|\sL(\cX \setminus \cU)|$, and let $D \subset \sL(Y)$ be a cylinder. Assume that for every field extension $k'$ of $k$, the map $\overline{\cC}(k') \to D(k')$ induced by $\cX \to Y$ is a bijection.
\begin{enumerate}[label=(\alph*)]

\item The restriction of $\het_{\cX/Y}$ to $\cC$ is integer valued and takes only finitely many values, and the set $\het_{\cX/Y}^{-1}(n) \cap \cC$ is a bounded cylinder for all $n \in \Z$.

\item We have the equality
\[
	\mu_Y(D) = \int_{\cC} \bL^{-\het_{\cX/Y}}\diff\mu_\cX.
\]

\end{enumerate}
\end{theorem}

\begin{proof}
Since $\cC$ is bounded, we may replace $\cX$ with a finite type open substack $\cX'$ containing $\theta_0(\cC)$ and replace $\cU$ with its intersection with $\cX'$ and therefore immediately reduce to the case where $\cX$ is finite type. Next, note that the proof of \cite[Corollary 2.2]{SatrianoUsatine3} never uses that $k$ is algebraically closed or characteristic 0, so it holds in our setting. Note that either $\cU$ is empty, in which case the theorem is vacuously true, or $\dim\cX = \dim Y$. We are thus reduced to showing that the change of variables formula \cite[Theorem 1.2]{SatrianoUsatine2} holds as stated without any assumptions on the field $k$. The proof of \cite[Theorem 1.2]{SatrianoUsatine2} in loc. cit. never uses that $k$ is algebraically closed, and the only place it uses that $k$ is characteristic 0 is in the proof of \cite[Proposition 6.1]{SatrianoUsatine2}. Therefore replacing \cite[Proposition 6.1]{SatrianoUsatine2} with its generalization \autoref{prop:paper2Prop6.1}, we are done.
\end{proof}

\subsection{Measurable and thin sets}

Measurable sets of arcs will be those sets of arcs that are approximated by bounded cylinders in a precise way, as we will now define.

\begin{definition}
Let $\cX$ be a smooth equidimensional Artin stack over $k$ with affine geometric stabilizers, let $\cC \subset |\sL(\cX)|$, and let $\varepsilon \in \R_{>0}$. A \emph{bounded cylinderical $\varepsilon$-approximation} of $\cC$ is a pair $(\cC^{(0)}, \{\cC^{(i)}\}_{i \in I})$ such that
\begin{itemize}

\item $\cC^{(0)} \subset |\sL(\cX)|$ is a bounded cylinder,

\item $\{\cC^{(i)}\}_{i \in I}$ is a collection with each $\cC^{(i)} \subset |\sL(\cX)|$ a bounded cylinder,

\item $\Vert \mu_\cX(\cC^{(i)}) \Vert < \varepsilon$ for all $i \in I$, and

\item $(\cC \cup \cC^{(0)}) \setminus (\cC \cap \cC^{(0)}) \subset \bigcup_{i \in I} \cC^{(i)}$.

\end{itemize}
\end{definition}

\begin{definition}
Let $\cX$ be a smooth equidimensional Artin stack over $k$ with affine geometric stabilizers, and let $\cC \subset |\sL(\cX)|$. The set $\cC$ is called \emph{measurable} if it has a bounded cylindrical $\varepsilon$-approximation for all $\varepsilon \in \R_{>0}$.
\end{definition}

\begin{lemma}\label{lemmaUntwistedBoundedCylinderFiniteSubcover}
Let $\cX$ be a locally finite type Artin stack over $k$, and let $\cC \subset |\sL(\cX)|$ be a bounded cylinder. Then any cover of $\cC$ by bounded cylinders has a finite subcover.
\end{lemma}

\begin{proof}
Since $\cC$ is bounded, we may replace $\cX$ with a finite type open substack containing $\theta_0(\cC)$ and therefore immediately reduce to the case where $\cX$ is finite type. Since $\cX$ is finite type, $|\sL(\cX)| \setminus \cC$ is a bounded cylinder, so we are reduced to the case where $\cC = |\sL(\cX)|$, which is \cite[Proposition 2.2]{SatrianoUsatine2}. We note that the proof of loc. cit. never uses any hypotheses on $k$.
\end{proof}

\begin{proposition}\label{propositionUntwistedMeasurableSetHasWellDefinedVolume}
Let $\cX$ be a smooth equidimensional Artin stack over $k$ with affine geometric stabilizers, and let $\cC \subset |\sL(\cX)|$ be a measurable set. There exists a unique element $\mu_{\cX}(\cC) \in \widehat{\sM}_k$ satisfying $\Vert \mu_\cX(\cC) - \mu_\cX(\cC^{(0)}) \Vert < \varepsilon$ for any bounded cylindrical $\varepsilon$-approximation $(\cC^{(0)}, \{\cC^{(i)}\}_{i \in I})$.
\end{proposition}

\begin{proof}
The proof in \cite[Chapter 6 Theorem 3.3.2]{ChambertLoirNicaiseSebag} of the special case where $\cX$ is a scheme works verbatim, except we need to use \autoref{propositionDimensionAndGrothendieckNorm} and \autoref{lemmaUntwistedBoundedCylinderFiniteSubcover} in place of their special cases where $\cX$ is a scheme.
\end{proof}

\begin{definition}
Let $\cX$ be a smooth equidimensional Artin stack over $k$ with affine geometric stabilizers, and let $\cC \subset |\sL(\cX)|$ be a measurable set. The \emph{motivic volume} of $\cC$ is the element $\mu_{\cX}(\cC)$ in the statement of \autoref{propositionUntwistedMeasurableSetHasWellDefinedVolume}.
\end{definition}

We now end this section with the following theorem about sets of arcs that factor through a closed substack.

\begin{theorem}\label{theoremUntwistedThinSubsets}
Let $\cX$ be a smooth equidimensional finite type Artin stack over $k$ with affine geometric stabilizers, and let $\cZ$ be a closed substack of $\cX$ with $\dim\cZ < \dim\cX$. Then $|\sL(\cZ)|$ is a measurable subset of $|\sL(\cX)|$ with $\mu_\cX(|\sL(\cZ)|) = 0$.
\end{theorem}

\begin{proof}
When $k$ is algebraically closed and characteristic 0, this is \cite[Theorem 9.2]{SatrianoUsatine6}. The proof in loc. cit. works verbatim without any hypotheses on $k$ as long as we use \autoref{propositionDimensionAndGrothendieckNorm} in place of \cite[Proposition 4.4]{SatrianoUsatine5}.
\end{proof}

\section{Warping stacks}

We will eventually generalize the theory developed in \autoref{sectionUntwistedMotivicIntegration} to the setting of \emph{twisted} arcs. As in the approach taken in \cite{SatrianoUsatine6}, we will reduce the main structural theorems of motivic integration over twisted arcs to the theory for untwisted arcs. The key technical machinery used to make this reduction is the notion of \emph{warped maps} and the \emph{warping stacks} that parametrize these. In this short section, we recall these notions and the main results about them that we will need. We begin with the following definition from \cite{SatrianoUsatine4}.

\begin{definition}
Let $\cX$ be an Artin stack over $k$. If $T$ is a scheme over $k$, a \emph{warped map} from $T$ to $\cX$ is a pair $(\cT \to T, \cT \to \cX)$, where $\cT$ is an Artin stack, $\cT \to T$ is a flat, finitely presented, good moduli space map with affine diagonal, and $\cT \to \cX$ is a representable morphism. Warped maps to $\cX$ form a category $\sW(\cX)$ fibered in groupoids over the category of $k$-schemes. The groupoid $\sW(\cX)(T)$ over $T$ is the category of warped maps from $T$ to $\cX$, and the morphisms in $\sW(\cX)$ are defined as in \cite[Remark 1.4]{SatrianoUsatine4}.
\end{definition}

We recall the following main structural result that will allow us to apply the theory in \autoref{sectionUntwistedMotivicIntegration} to the warping stack $\sW(\cX)$.

\begin{theorem}
Let $\cX$ be a locally finite type Artin stack over $k$ with affine diagonal. Then $\sW(\cX)$ is a locally finite type Artin stack over $k$. Furthermore if $\cX$ has a good moduli space, then $\sW(\cX)$ has separated diagonal and affine geometric stabilizers.
\end{theorem}

\begin{proof}
The first claim is \cite[Theorem 1.9(1)]{SatrianoUsatine4}, and the second claim is \cite[Proposition 8.1]{SatrianoUsatine6}. We note that although both results were stated with the assumption that $k$ is algebraically closed and characteristic 0, these hypotheses were not used in the proofs.
\end{proof}

We will also need a certain open substack $\widetilde{\sW}(\cX)$ of $\sW(\cX)$ introduced in \cite{SatrianoUsatine6}, as we now recall.

\begin{notation}
If $\cX$ is an Artin stack over $k$, we let $\sW^{\synt}(\cX)$ denoted the full subcategory of $\sW(\cX)$ whose objects are warped maps $(\cT \to T, \cT \to \cX)$ where $\cT \to T$ is syntomic.
\end{notation}

\begin{proposition}\label{propositionSyntomicLocusOpenAndSmooth}
Let $\cX$ be a smooth Artin stack over $k$ with affine diagonal. Then $\sW^\synt(\cX)$ is an open substack of $\sW(\cX)$, and $\sW^\synt(\cX)$ is smooth over $k$.
\end{proposition}

\begin{proof}
This is \cite[Proposition 5.2]{SatrianoUsatine6}, whose proof in loc. cit. never uses any hypotheses on the field $k$.
\end{proof}

For any Artin stack $\cX$ over $k$, there is a canonical morphism $\cX \to \sW^{\synt}(\cX)$ given by sending each map $T \to \cX$ to the warped map $(T \xrightarrow{\id_T} T, T \to \cX)$.

\begin{proposition}\label{propositionClosureInSyntomicLocus}
Let $\cX$ be a smooth Artin stack over $k$ with affine diagonal. Then the canonical morphism $\cX \to \sW^{\synt}(\cX)$ is an open immersion, and the closure of the image of $|\cX| \to |\sW^\synt(\cX)|$ in $|\sW^\synt(\cX)|$ is an open subset of $|\sW^\synt(\cX)|$.
\end{proposition}

\begin{proof}
The first claim is immediate from \autoref{propositionSyntomicLocusOpenAndSmooth} and \cite[Theorem 1.9(2)]{SatrianoUsatine4}, and the second claim is \cite[Proposition 5.4]{SatrianoUsatine6}. The proofs of \cite[Theorem 1.9(2)]{SatrianoUsatine4} and \cite[Proposition 5.4]{SatrianoUsatine6} work verbatim for general $k$.
\end{proof}

\autoref{propositionClosureInSyntomicLocus} allows us to set the following notation.

\begin{notation}
If $\cX$ is a smooth Artin stack over $k$ with affine diagonal, we will let $\widetilde{\sW}(\cX)$ denote the open substack of $\sW^{\synt}(\cX)$ whose support is the closure of the image of $|\cX| \to |\sW^\synt(\cX)|$ in $|\sW^\synt(\cX)|$.
\end{notation}

\begin{remark}
By the above results, if $\cX$ is a smooth equidimensional Artin stack over $k$ with affine diagonal and $\cX$ has a good moduli space, then $\widetilde{\sW}(\cX)$ is a smooth equidimensional Artin stack over $k$ with affine geometric stabilizers and separated diagonal, and $\dim \widetilde{\sW}(\cX) = \dim \cX$. In particular, $\widetilde{\sW}(\cX)$ satisfies the hypotheses of \autoref{definitionMotivicVolumeUntwistedArcs}. If furthermore $\cX$ is irreducible, then $\widetilde{\sW}(\cX)$ is also irreducible, so $\widetilde{\sW}(\cX)$ satisfies the hypotheses for the source Artin stack in \autoref{theoremUntwistedChangeOfVariables}.
\end{remark}

\section{Motivic integration over twisted arcs}

In this short section, we will recall the basics of motivic integration over twisted arcs of Artin stacks as defined in \cite{SatrianoUsatine6}. We begin by setting notation for \emph{twisted discs} and their truncations.

\begin{notation}
If $A$ is a ring and $r \in \Z_{>0}$, we will set
\[
	\cD^r_A = [\Spec(A\llbracket t^{1/r} \rrbracket) / \mu_r],
\]
where the action of the $r$-th roots of unity $\mu_r$ on $\Spec(A\llbracket t^{1/r} \rrbracket)$ is given by $\xi \in \mu_r$ taking $f(t^{1/r})$ to $f(\xi t^{1/r})$. In the special case $r = 1$, we also use the notation $D_A = \cD^1_A = \Spec(A\llbracket t \rrbracket)$. We let $\cD^r_A \to D_A$ be the usual coarse space map. We also let $D_A \to \cD^r_A$ denote the fppf cover given by precomposing the quotient map $\Spec(A\llbracket t^{1/r}\rrbracket) \to \cD^r$ with the isomorphism $D_A \xrightarrow{\sim} \Spec(A\llbracket t^{1/r}\rrbracket)$ given by $t^{1/r} \to t$. For any $n \in \Z_{\geq 0}$, we will set
\[
	\cD^r_{n, A} = \cD^r_A \otimes_{A\llbracket t \rrbracket} A[t]/(t^{n+1}).
\]
In the special case $r = 1$, we also use the notation $D_{n,A} = \cD^1_{n,A} = \Spec(A[t]/(t^{n+1}))$. We thus have coarse space maps $\cD^r_{n,A} \to D_{n, A}$ and fppf covers $D_{r(n+1) - 1,A} \to \cD^r_{n,A}$. When it is clear from context, especially when $A = k$, we will sometimes suppress $A$ in the notation and use $\cD^r, D, \cD^r_n, D_n$. 
\end{notation}

We now set notation for \emph{twisted jet stacks} and \emph{twisted arc stacks} as in \cite[Section 4]{SatrianoUsatine6}.

\begin{notation}
If $r \in \Z_{>0}$, $n \in \Z_{\geq 0}$, and $\cX$ is a finite type Artin stack over $k$ with affine diagonal, we let $\sJ^r_n(\cX)$ denote the Hom stack parametrizing representable $k$-morphisms from $\cD^r_{n, k}$ to $\cX$. For $m \geq n$, we have the \emph{truncation map} $\theta^m_n: \sJ^r_m(\cX) \to \sJ^r_n(\cX)$ induced by the closed immersion $\cD^r_n \hookrightarrow \cD^r_m$. We then set $\sJ^r(\cX) = \varprojlim_n \sJ^r_n$, and we let $\theta_n: \sJ^r(\cX) \to \sJ^r_n(\cX)$ denote the canonical map. We also set $\sJ_n(\cX) = \bigsqcup_{r \in \Z_{>0}} \sJ^r_n(\cX)$ and $\sJ(\cX) = \bigsqcup_{r \in \Z_{\geq 0}} \sJ^r(\cX)$, and we let $\theta_n: \sJ(\cX) \to \sJ_n(\cX)$ denote the obvious map.
\end{notation}

\begin{remark}
By the exact same reasoning as in \cite[Remark 4.4]{SatrianoUsatine6}, each $\sJ^r_n(\cX)$ is a finite type Artin stack over $k$ with affine diagonal.
\end{remark}

\begin{remark}\label{remarkTwistedArcsAreTwistedArcs}
By the exact same reasoning as in \cite[4.7]{SatrianoUsatine6}, for any field extension $k'$ of $k$, we have $\sJ^r(\cX)(k')$ is canonically equivalent to the category of representable $k$-morphisms $\cD^r_{k'} \to \cX$.
\end{remark}

As in the untwisted case, we define the following special subsets of $|\sJ(\cX)|$.

\begin{definition}
Let $\cX$ be a finite type Artin stack over $k$ with affine diagonal, and let $\cC \subset |\sL(\cX)|$. The set $\cC$ is called a \emph{cylinder} if there exists some $n \in \Z_{\geq 0}$ and locally constructible subset $\cC_n \subset |\sJ_n(\cX)|$ such that $\cC = \theta_n^{-1}(\cC_n)$. The set $\cC$ is called \emph{bounded} if $\theta_0(\cC)$ is contained in a quasi-compact subset of $|\sJ_0(\cX)|$, or equivalently, if $\cC \cap |\sJ^r(\cX)| = \emptyset$ for all but finitely many $r$.
\end{definition}

The following allows us to define the motivic volume of bounded cylinders.

\begin{theorem}
Let $\cX$ be a smooth finite type equidimensional Artin stack over $k$ with affine diagonal and a good moduli space, and let $\cC \subset |\sJ(\cX)|$ be a bounded cylinder. Then $\theta_n(\cC)$ is a quasi-compact locally constructible subset of $|\sJ_n(\cX)|$ for all $n \in \Z_{\geq 0}$, and the sequence
\[
	\{\e(\theta_n(\cC))\bL^{-(n+1)\dim\cX}\}_{n \in \Z_{\geq 0}} \subset \widehat{\sM}_k
\]
stabilizes for $n$ sufficiently large.
\end{theorem}

\begin{proof}
Since $\cC$ is bounded, we can immediately reduce to the case where $\cC \subset |\sJ^r(\cX)|$ for some $r \in \Z_{>0}$. This case is \cite[Theorem 4.12]{SatrianoUsatine6}, whose proof in loc. cit. never uses that $k$ is algebraically closed or characteristic 0.
\end{proof}

We may therefore make the following definition.

\begin{definition}
Let $\cX$ be a smooth finite type equidimensional Artin stack over $k$ with affine diagonal and a good moduli space, and let $\cC \subset |\sJ(\cX)|$ be a bounded cylinder. Then the \emph{motivic volume} of $\cC$ is defined as
\[
	\nu_\cX(\cC) = \lim_{n \to \infty} \e(\theta_n(\cC))\bL^{-(n+1)\dim\cX} \in \widehat{\sM}_k.
\]
\end{definition}

We end this section by briefly recalling the height and weight functions on $\sJ(\cX)$, introduced in \cite{SatrianoUsatine6}, that will appear in the motivic change of variables formula for twisted arcs.

\begin{notation}
Let $\cX$ be a smooth finite type Artin stack over $k$ with affine diagonal, let $\cY$ be a locally finite type Artin stack over $k$, and let $\cX \to \cY$ be a morphism. If $\varphi: \cD^r_{k'} \to \cX$ is a twisted arc of $\cX$ for some field extension $k'$ of $k$, then we set
\[
	\het_{\cX/\cY}(\varphi) = (1/r)\het_{\cX/\cY}(\psi),
\]
where $\psi: D_{k'} \to \cX$ is the composition of the fppf cover $D_{k'} \to \cD_{k'}$ with $\varphi$. This induces a map
\[
	\het_{\cX/Y}: |\sJ(\cX)| \to \Q \cup \{\infty\}.
\]
\end{notation}

\begin{notation}
Let $\cX$ be a finite type Artin stack over $k$ with affine diagonal. We have a map $\sJ^r_0(\cX) \to I_{\mu_r}(\cX)$ induced by the closed immersion $B\mu_r \hookrightarrow \cD^r_0$, so we can consider the map
\[
	\wt_\cX: |\sJ(\cX)| \xrightarrow{\theta_0} |\sJ_0(\cX)| \to |I_\mu(\cX)| \xrightarrow{\overline{\wt}_\cX} \Q,
\]
where $\overline{\wt}_\cX$ is the weight function defined in \autoref{definitionCyclotomicInertiaAndWeightFunction}.
\end{notation}

As we did in the untwisted setting, we will eventually obtain a change of variables formula and thin subset result in the twisted setting. We will do so by using warping stacks to reduce these to the untwisted setting, and therefore we will need a comparison between the motivic measures on $|\sJ(\cX)|$ and $|\sL(\sW(\cX))|$. We will obtain such a comparison in \autoref{sectionTwistedToWarpedComparison} below, but in order to do so, we will need to first generalize some technical statements from \cite{SatrianoUsatine6} to the setting where $k$ has positive characteristic. The generalizations of these technical statements require significant changes to the proofs (and sometimes to the statements themselves) in order to work in positive characteristic, and we accomplish this in the next two sections.

\section{Twisted forms of twisted discs and their truncations}

The goal of this section is to prove the following generalization of \cite[Proposition 6.1]{SatrianoUsatine6}, which shows that twisted discs and their truncations have no nontrivial twisted forms. Specifically, \autoref{prop:no-non-trivial-forms-truncated-twisted-charp} removes the assumption that $r$ is prime to the characteristic and the requirement that $L/K$ is separable.

\begin{proposition}\label{prop:no-non-trivial-forms-truncated-twisted-charp}
Let $L/K$ be a finite extension of fields and assume $r\geq2$. Let $n\in\bZ\cup\{\infty\}$ and assume $n\geq2$. If $\cT_n\to D_{n,K}$ is a map from an Artin stack and we have a $D_{n,L}$-isomorphism $\cT_n\times_K L\xrightarrow{\simeq}\cD_{n,L}^{r}$, then there exists a $D_{n,K}$-isomorphism $\cT_n\xrightarrow{\simeq}\cD_{n,K}^{r}$; here $\cD_\infty^{r}:=\cD^{r}$.
\end{proposition}

We begin by proving an analogue of \cite[Lemma 6.2]{SatrianoUsatine6}. 

\begin{lemma}\label{l:twisted-twisted-discs-def-sp-redone}
Fix $r\in\bZ_{>1}$ and let $A$ be a ring whose characteristic divides $r$. Let $n\in\bZ_{\geq0}$, $\pi_n\colon\cD_{n,A}^{r}\to D_{n,A}$ be the coarse space map, and $J=(t^{n+1})/(t^{n+2})$. Then for all $i\neq0,1$,
\[
\Ext^i(L_{\cD^{r}_{n,A} / D_{n,A}},\pi_n^*J)=0.
\]

Furthermore, letting $S'=A[u,v^\pm]/I'$ with $I'=(u^r-tv,t^{n+1})$, then $\cD^{r}_{n,A}=[\Spec S'/\bG_m]$ where $\bG_m$ acts on $u,v$ with weights $1,r$; we additionally have natural isomorphisms
%\[\Ext^0(L_{\cD^{r}_{n,A} / D_{n,A}},\pi_n^*J)\simeq\Hom(S\otimes_{A[u]}\Omega^1_{A[u]/A},\pi_n^*J)^{\mu_r}\]and
\[
\Ext^1(L_{\cD^{r}_{n,A} / D_{n,A}},\pi_n^*J)\simeq\Hom(I'/(I')^2,\pi_n^*J)^{\bG_m}
\]
and
\[
\Ext^0(L_{\cD^{r}_{n,A} / D_{n,A}},\pi_n^*J)\simeq\pi_n^*J.
\]
\end{lemma}
\begin{proof}
We suppress the subscript $A$ throughout the proof. To see $\cD^{r}_{n,A}=[\Spec S'/\bG_m]$, we simply observe that $\cD^{r}_{n,A}=[\Spec S/\mu_r]$ where $S=A[t^{1/r}]/t^{n+1}$, and we may rewrite the stack as $[\Spec S'/\bG_m]$ where $S'=(S[w^\pm])^{\mu_r}$ where $\mu_r$ acts on $t^{1/r},w$ with weights $1,-1$. We then see $S'=A[u,v^\pm]/I'$ where $v=w^r$, $u=t^{1/r}w$.

The same argument as in the proof of \cite[Lemma 6.2]{SatrianoUsatine6} allows us to reduce to the case where $n=0$. Letting $\rho\colon\Spec S'\to\cD^{r}_0$ be the smooth cover, we have an exact triangle
\[
\rho^*L_{\cD^{r}_0/A}\to L_{S'/A}\to L_{S'/\cD^{r}_0}.
\]
Since $L_{S'/\cD^{r}_0}=\cO_{S'}\otimes\g^\vee$ where $\g$ is the Lie algebra of $\bG_m$, we obtain a short exact sequence
\[
0\to\Ext^{-1}(L_{\cD^{r}_0 / A},\cO_{\cD^{r}_0})\to (\cO_{S'}\otimes\g)^{\bG_m}\xrightarrow{\alpha}\Ext^0(L_{S'/A},\cO_{S'})^{\bG_m}\to \Ext^0(L_{\cD^{r}_0 / A},\cO_{\cD^{r}_0})\to 0
\]
and for $i\notin[-1,0]$ we have
\[
\Ext^i(L_{\cD^{r}_0 / A},\cO_{\cD^{r}_0})=\Ext^i(L_{S'/A},\cO_{S'})^{\bG_m}.
\]
Since $A\to S'$ is lci, we see the latter quantity vanishes for $i\neq0,1$. In particular, we see
\[
\Ext^i(L_{\cD^{r}_0 / A},\cO_{\cD^{r}_0})=0
\]
if $i\notin[-1,1]$.

We now calculate $\Ext^1(L_{\cD^{r}_0 / A},\cO_{\cD^{r}_0})=\Ext^1(L_{S'/A},S')^{\bG_m}$. For this, we consider the regular closed immersion $\Spec S'\to\Spec A[u,v^\pm]$. 
%From this, a computation analogous to that in Lemma 6.2 shows we have an exact sequence
From this we obtain an exact triangle
\[
I'/(I')^2\to \Omega^1_{A[u,v^\pm]/A}|_{S'}\to L_{S'/A}.
\]
Note $I'=(u^r)$ and that the left hand map maps the generator to $d(u^r)=r u^{r-1}du=0$ since $r=0$ in $A$. Thus,
\[
\Ext^1(L_{S'/A},S')^{\bG_m}=\Hom(I'/(I')^2,S')^{\bG_m}.
\]
Since $I'/(I')^2=S'$, we see
\[
\Ext^1(L_{S'/A},S')^{\bG_m}=S'_0=A
\]
where the subscript $0$ denotes the degree $0$ part.

Lastly, we compute $\Ext^i(L_{\cD^{r}_0 / A},\cO_{\cD^{r}_0})$ for $i=-1,0$ by computing the kernel and cokernel of the map $\alpha$ above. Recall that $\Ext^0(L_{S'/A},\cO_{S'})=\Hom(\Omega^1_{S'/A},S')$. The map
\[
\Omega^1_{S'/A}\to \cO_{S'}\otimes\g^\vee\simeq\cO_{S'}
\]
is given by $du\mapsto u$ and $dv\mapsto r v=0$; note that $\bG_m$-action $\g$ is trivial as $\bG_m$ is abelian. Dualizing and taking $\bG_m$-invariants, we obtain
\[
\alpha\colon A\to \Hom(\Omega^1_{S'/A},S')^{\bG_m}\simeq Au\oplus Av,\quad\alpha(1)=(u,0);
\]
the computation of $\Hom(\Omega^1_{S'/A},S')^{\bG_m}$ follows from the fact that $\Hom(\Omega^1_{S'/A},S')$ is free with generators $\partial_u,\partial_v$ which have weights $-1,-r$. As a result,
\[
\Ext^{-1}(L_{\cD^{r}_0 / A},\cO_{\cD^{r}_0})=\ker(\alpha)=0
\]
and 
\[
\Ext^0(L_{\cD^{r}_0 / A},\cO_{\cD^{r}_0})=\coker(\alpha)=Av\simeq A.\qedhere
\]
\end{proof}

%\matt{Note that we may assume $r>1$. Furthermore, the only computations needed for Section 6 are that $\Ext^2=0$ (used in the proof of 6.1) and the specific computation of $\Ext^1$ (used in the proof of 6.3).}

We next generalize \cite[Proposition 6.3]{SatrianoUsatine6}.

\begin{proposition}\label{prop:twisted-twisted-discs-defs-truncated-twisted-charp}
Fix $r\in\bZ_{\geq2}$, $n\in\bZ_{\geq2}$, and $\pi_n\colon\cD_{n,A}^{r}\to D_{n,A}$ be the coarse space map where $A$ is a ring whose characteristic divides $r$. If
\[
\xymatrix{
\cD_{n,A}^{r}\ar[r]\ar[d]_-{\pi_n} & \cC\ar[d]^-{q}\\
D_{n,A}\ar[r] & D_{n+1,A}
}
\]
is a cartesian diagram with $q$ flat, then we have a $D_{n+1,A}$-isomorphism $\sigma\colon\cC\xrightarrow{\simeq}\cD_{n+1,A}^{r}$ such that $\sigma\times_{D_{n+1,A}} D_{n-1,A}$ is the identity automorphism.
\end{proposition}
\begin{proof}
The proof is nearly the same as that of \cite[Proposition 6.3]{SatrianoUsatine6}. We suppress the subscript $A$ throughout the proof, let $R_m:=A[t]/t^{m+1}$, i.e., the coordinate ring of $D_m$. Then replacing the use of \cite[Lemma 6.2]{SatrianoUsatine6} with \autoref{l:twisted-twisted-discs-def-sp-redone} shows that deformations of $\cD_n^r$ are given by $\bG_m$-equivariant deformations of
\[
\Spec (R_n[u,v^\pm]/(u^r v^{-1}-t).
\]
Hence, $\cC\simeq [C/\bG_m]$ where 
\[
C=\Spec R_{n+1}[u,v^\pm]/(u^r v^{-1}-(t+rt^{n+1}))
\]
and $r\in R_{n+1}$. Lastly, note that $v\mapsto (1+rt^n)v$ defines a $\bG_m$-equivariant isomorphism $R_{n+1}[u,v^\pm]/(u^r v^{-1}-t)$ and $C$, and that this is the identity automorphism mod $t^n$.
%Indeed, letting $\rho\colon D_m\to\cD^{r}_n$ be the standard \'etale cover with $m=r(n+1)-1$, \autoref{l:twisted-twisted-discs-def-sp} and its proof show (following notation from the lemma) $\Ext^1(L_{\cD^{r}_n / D_n},\pi_n^*J)=\Ext^1(L_{D_m / D_n},\rho^*\pi_n^*J)^{\mu_r}=\Hom(I/I^2,\pi_n^*J)^{\mu_r}$. The first equality tells us that $\cC$ is given by an equivariant deformation of $D_m$, i.e., $\cC=[C/\mu_r]$ with $C$ an equivariant deformation of $D_m$; the second equality tells us $C$ is given by deforming the equation $u^r-t$ to one of the form $u^r - t - ct^{n+1}$ with $c\in A$.
\end{proof}

We now prove the main result of this section.

\begin{proof}[{Proof of \autoref{prop:no-non-trivial-forms-truncated-twisted-charp}}]
If $\cha(K)$ divides $r$, then in light of \autoref{l:twisted-twisted-discs-def-sp-redone} and \autoref{prop:twisted-twisted-discs-defs-truncated-twisted-charp}, the proof of \cite[Proposition 6.1]{SatrianoUsatine6} goes through exactly the same to reduce us to the case where $n=\infty$, i.e., we must show there are no twisted forms of $\cD_K^r$. If $\cha(K)$ is prime to $r$, then the proof as written in \cite[Proposition 6.1]{SatrianoUsatine6} also reduces us to the case $n=\infty$ without yet making use of the assumption that $L/K$ is separable.

The rest of this proof makes no assumptions on $\cha(K)$. The proof of \cite[Lemma 6.6]{SatrianoUsatine6} handles the $n=\infty$, however it assumes $L/K$ is separable. In order to drop this hypothesis we note that separability of $L/K$ was only used in two places:~(i) to conclude $K((t))\otimes_K L=L((t))$ and (ii) since $D_L\to\cD_L^r=\cT_\infty\times_K L\to\cT_\infty$ is an \'etale cover and $D_L$ is regular, we concluded $\cT_\infty$ is regular. For (i), we note that the proof of \cite[Lemma 6.6]{SatrianoUsatine6} never actually used $K((t))\otimes_K L=L((t))$ but instead only used $K((t))\otimes_{K[[t]]} L[[t]]=L((t))$ which holds without the separability hypothesis. For (ii), even when $L/K$ is not separable, $\Spec L\to\Spec K$ is an fppf cover and so regularity of $D_L$ implies regularity of $\cT_\infty$ by \autoref{l:fppf-loc-sm->sm}.
\end{proof}

We end this section with a lemma that will not be used until \autoref{sectionTwistedChangeOfVariables} below, but we include it here because its proof is similar to the ideas above.

\begin{lemma}\label{lemmaExtComputationTwistedDiscOverDisc}
Fix $r\in\bZ_{>1}$ and let $k'$ be a field. If $\pi\colon\cD\to D=\Spec k'[[t]]$ is the $r$-th twisted disc, then
\[
\dim\Ext^1(L_{\cD/D},\cO)=1
\]
and
\[
\Ext^i(L_{\cD/D},\cO)=0
\]
for $i\neq 1$.
\end{lemma}
\begin{proof}
If $\cha(k')$ does not divide $r$, then the proof of \cite[Lemma 5.2]{SatrianoUsatine6} shows the result. So we may assume $\cha(k')$ divides $r$. Let $S=k[[t]][u,v^\pm]/I$ where $I=(u^r-tv)$, and let $\bG_m$ act on $W=\Spec S$ with weights $1,r$ on $u,v$. Then $\cD=[W/\bG_m]$; let $\rho\colon W\to\cD$ be the smooth cover. We have a distinguished triangle
\[
\rho^*L_{\cD/D}\to L_{W/D}\to \cO_W\otimes_{k'}\g^\vee
\]
where $\g$ is the Lie algebra of $\bG_m$. Hence we see
\[
\Ext^i(L_{\cD/D},\cO)=\Ext^i(L_{W/D},\cO)^{\bG_m}
\]
for all $i\neq -1,0$, and we have an exact sequence
\[
0\to\Ext^{-1}(L_{\cD/D},\cO_W)\to \cO_D\otimes\g\xrightarrow{\alpha}  \Hom(\Omega^1_{W/D},\cO)^{\bG_m}\to \Ext^0(L_{\cD/D},\cO_W)\to 0
\]
where we have used $\cO_D\otimes\g=(\cO_W\otimes\g)^{\bG_m}$ and $\Ext^0(L_{W/D},\cO)=\Hom(\Omega^1_{W/D},\cO)$. Since $W\to D$ is lci, we find $\Ext^i(L_{\cD/D},\cO)=0$ for all $i\neq -1,0,1$. Note that $\Omega^1_{S/k'[[t]]}$ is generated by $du$, $dv$ subject to the relation $tdv=0$. Since $t$ is a non-zero divisor in $S$, we see $\Hom(\Omega^1_{S/k'[[t]]},S)$ is generated by the functional $\partial_u$ dual to $du$, and $\alpha$ is given by
\[
\alpha\colon S\simeq k'[[t]]\otimes\g\to \Hom(\Omega^1_{S/k'[[t]]},\cO)^{\bG_m}\simeq k'[[t]] u\partial_u,\quad\alpha(1)=u\partial_u.
\]
Thus, $\alpha$ is an isomorphism, showing that $\Ext^i(L_{\cD/D},\cO_W)=0$ for $i=-1,0$.

It remains to prove $\dim\Ext^1(L_{\cD/D},\cO_W)=\dim\Ext^1(L_{S/k'[[t]]},S)^{\bG_m}=1$. Using that $W\to D$ is lci and letting $S'=k'[[t]][u,v^\pm]$, we see
\[
\Hom(S\otimes_{S'}\Omega^1_{S'/k'[[t]]},S)\xrightarrow{\beta} \Hom(I/I^2,S)\to\Ext^1(L_{S/k'[[t]]},S)\to0
\]
is exact. The map
\[
I/I^2\to S\otimes_{S'}\Omega^1_{S'/k'[[t]]}
\]
is given by $f=u^r-tv\mapsto d(u^r)-tdv=-tdv$, so $\beta(\partial_u)=0$ and $\beta(\partial_v)=-t\partial_f$. Thus,
\[
\Ext^1(L_{S/k'[[t]]},S)^{\bG_m}=(S/t)^{\bG_m}=(k'[u,v^\pm]/u^r)^{\bG_m}=k'
\]
which finishes the proof.
\end{proof}

\section{Fibers of the twisted-to-warped map}\label{sectionTwistedToWarpedFibers}

For the remainder of this paper, if $r \in \Z_{> 0}$, $n \in \Z_{\geq 0}$, and $\cX$ is a finite type Artin stack over $k$ with affine diagonal, we will let $\sJ^r_n(\cX) \to \sL_n(\sW(\cX))$ denote the obvious map taking the twisted jet $\cD^r_n \to \cX$ to the warped jet $(\cD^r_n \to D_n, \cD^r_n \to \cX)$. In this section, we prove the following result, which generalizes \cite[Theorem 7.1]{SatrianoUsatine6} and computes the fibers of the map $\sJ^r_n(\cX) \to \sL_n(\sW(\cX))$ in the setting of interest.

\begin{theorem}\label{thm:fibers-twisted->warped-aut-explicit-charp}
Let $\cX$ be a finite type Artin stack with affine diagonal over a field $k$. Suppose either
\begin{enumerate}[label=(\alph*)]
\item $\cha(k)=0$ and $n,r\geq2$, or
\item $\cha(k)=p$, $r=p^s r'$ with $\gcd(p,r')=1$, and $n$ is divisible by $p^s$.
\end{enumerate}
Let $L/k$ be a field extension, and consider an $L$-point of $\sL_n(\sW(\cX))$ given by the warped map $(\cD_n\to D_{n,L},f\colon\cD_n\to\cX)$.

If the fiber of $\cJ^{r}_n(\cX)\to \sL_n(\sW(\cX))$ is non-empty, then it is given by
\[
\cJ^{r}_n(\cX)\times_{\sL_n(\sW(\cX))} \Spec L\simeq \bA^1_L.
\]
More specifically, the fiber is a $\bG_a$-torsor, where $\bG_a$ is canonically the vector group associated to the $1$-dimensional vector space
\[
\Ext^0(L_{\cD^{r}_{0,L} / L},\mcJ);
\]
here $\mcJ$ denotes the pull back to $\cD^{r}_{0,L}$ of the ideal sheaf defining the closed immersion $D_{n,L}\hookrightarrow D_{n+1,L}$.
\end{theorem}

%First of all, note the proof of 7.2 uses 6.1 and here $k''/k'$ is a finite extension which is not necessarily separable. BUT WE GENERALIZED 6.1 TO HOLD W/O SEPARABILITY SO THIS IS ALL OK. Alternatively there is a commented out argument which I leave here in case we need it. %However, we can get away with the \'etale hypothesis in 6.1 b/c........ Suppose we first prove Theorem 7.1 in the special case where L is perfect. Then I think in the proof of 7.2, we will have $k' = L$ is perfect, so $k''/k'$ is separable, so then maybe it's fine to keep separable in 6.1. Then we want the case where L is not perfect. But now we know that the fiber $\cF$ is a (possibly) twisted form of affine space, so it's smooth (fpqc locally smooth), we then \'etale locally have a section, so now $\cF$ has a $k''$ point with $k''$ a finite separable extension of $k'$. And then we can apply 6.1 with the separable hypothesis remaining.

We begin with a helpful lemma.

\begin{lemma}\label{l:ell-th-root}
Suppose either
\begin{enumerate}[label=(\alph*)]
\item $A$ is a $\QQ$-algebra and $r\geq2$, or
\item $A$ is an $\bF_p$-algebra with $r=p^s r'$ and $\gcd(p,r')=1$, and let $n$ be divisible by $p^s$.
\end{enumerate}
Then for every $a\in A$, there is an fppf cover $q\colon\Spec A'\to\Spec A$ such that $1+at^n$ has an $r$-th root in $A'[t]/(t^{n+1})$. If $A$ is a $\QQ$-algebra, we can take $q=\id$.
\end{lemma}
\begin{proof}
If $A$ is a $\QQ$-algebra, observe
\[
(1+ar^{-1}t^n)^r=1+at^n.
\]
Now suppose $A$ is an $\bF_p$-algebra. Let $A'=A[x]/(x^{p^s}-a)$. Then $\Spec A'\to\Spec A$ is an fppf cover and
\[
\left(1+\frac{xt^{n/p^s}}{r'}\right)^r=1+at^n
\]
in $A'[t]/(t^{n+1})$.
\end{proof}

The following result is the analogue of \cite[Proposition 7.7 and Lemma 7.9]{SatrianoUsatine6}.

\begin{proposition}\label{prop:auts-Dnell-reduce-to-id}
Let $S$ be a scheme over a field $k$ with
\begin{enumerate}[label=(\alph*)]
\item $\cha(k)=0$ and $n,r\geq2$, or
\item $\cha(k)=p$, $r=p^s r'$ with $\gcd(p,r')=1$, and $n$ is divisible by $p^s$.
\end{enumerate}
If $\sigma\in\Aut_{D_{n,S}}(\cD_{n,S}^r)$, then
\begin{enumerate}
\item\label{prop:auts-Dnell-reduce-to-id::lift} there is an fppf cover $q\colon S'\to S$ such that $\sigma\times_S S'$ is induced by an equivariant automorphism of the cover $D_{r(n+1)-1,S'}$ of $\cD_{n,S'}^r$. Furthermore, if $S=\Spec R$ and all elements of $R$ have $r$-th roots (e.g., $R$ is an algebraically closed field), then we may take $q=\id$.

\item\label{prop:auts-Dnell-reduce-to-id::2auts} the only $2$-automorphism of $\sigma\times_S S'$ is $\id$.

\item\label{prop:auts-Dnell-reduce-to-id::id} if $\tau$ denotes the image of $\sigma\times_S S'$ in $\Aut_{D_{n-1,S'}}(\cD_{n-1,S'}^r)$, then $\tau$ is isomorphic to $\id$.
\end{enumerate}
\end{proposition}
\begin{proof}
Replacing $S$ by an affine open cover, we may assume $S=\Spec A$; we suppress $A$ and $S$ from the notation throughout the proof. Let $m=r(n+1)-1$ and $\rho\colon D_m\to\cD_n^{r}$ be the standard $\mu_r$-torsor, which corresponding to the line bundle $\cO(1)$ with its trivialization $\iota\colon\cO\xrightarrow{\simeq}\cO(r)=\cO(1)^{\otimes r}$ given by the section $t\in B:=A[t]/t^{n+1}$. 
Consider the cartesian diagram
\[
\xymatrix{
P\ar[r]^-{f}_-{\simeq}\ar[d] & D_m\ar[d]^-{\rho}\\
\cD_n^r\ar[r]^-{\sigma}_-{\simeq} & \cD_n^r
}
\]
Then $P\to\cD_n^{r}$ corresponds to the line bundle $\sigma^*\cO(1)$ and the trivialization $\sigma^*\iota$. Looking at how $\sigma$ acts on the closed point of $\cD_n^r$, we see it induces an automorphism of $B\mu_r$ given by $\zeta\mapsto\zeta^a$ for some fixed $a$ with $\gcd(a,r)=1$. As a result, $(\sigma^*\cO(1))|_{B\mu_r}\simeq\cO_{B\mu_r}(a)\otimes\cM$ where $\cM$ is the pullback of a line bundle on $S$; after looking Zariski locally on $S$, we may assume $\cM=\cO$. Thus, $\sigma^*\cO(1)$ is a deformation of $\cO_{B\mu_r}(a)$. Since $\gcd(a,r)=1$, the map $q\colon B\mu_r\to B\bG_m$ induced by $\cO_{B\mu_r}(a)$ is representable. Thus, an argument as in \cite[Proposition 2.5]{SatrianoUsatine6} shows $\sigma^*\cO(1)\simeq\cO(a)$. More specifically, using induction by considering the nilpotent thickenings of $B\mu_r$ in $\cD_n^r$ and using the fact that $L_{B\bG_m}$ is concentrated in degree $1$, we see that the deformation space of $\cO_{B\mu_r}(a)$, namely $\Ext^0(Lq^*L_{B\bG_m},\cO)$, vanishes. Hence, there is a unique deformation of $\cO_{B\mu_r}(a)$ to  $\cD_n^r$. Since $\cO(a)$ is such a deformation, we see $\sigma^*\cO(1)\simeq\cO(a)$.

We next show $a=1$. Note that $m\geq1$ and $\sigma$ must preserve the first infinitesimal neighbourhood $\cZ:=[D_1/\mu_r]$ of the closed point of $\cD^r_n$. Since $\dim H^0(\cO_\cZ(1))\neq0$ and $\dim H^0(\cO_\cZ(b))=0$ for $1<b<r$, we see $0\leq a\leq1$. On the other hand, $\sigma^*\cO(1)$ is non-trivial so we see $a=1$.

Hence, the $\mu_r$-torsor $P\to\cD^r_n$ corresponds to the line bundle $\cO(1)$ with trivialization $\alpha\colon\cO\xrightarrow{\simeq}\cO(1)^{\otimes r}=\cO(r)$ given by $\alpha(1)=gt$ for some unit $g$ of $B$. Then 
\[
P=\Spec B[v]/(v^r-gt),
\]
where $v$ has weight $1$. Furthermore, $f$ is a $D_n$-isomorphism which is $\mu_r$-equivariant (since $\sigma$ induces the $a$-th power map on $\mu_r$ and $a=1$), so we must have $f(u)=hv$ with $h$ of weight $0$ and
\[
g^{-1}v^r=t = f(t)=h^r v^r.
\]
Therefore
\[
h^r g\in 1+\Ann(v^r)=1+(v^{nr})=1+(t^n).
\]
Thus,
\[
h^r g=1+ct^n
\]
for some $c\in B$.

By \autoref{l:ell-th-root}, after replacing $S$ by an fppf cover, we may assume there exists $h'\in B$ such that $(h')^r=1+ct^n$. Note that $h'/h$ is in $B$ since $h$ is a unit. Then consider the automorphism $\beta$ of $\cO(1)$ given by multiplication by $h'/h$. Then $\beta^{\otimes r}\circ\iota=\alpha$, thereby showing $P\to\cD^r_n$ and $D_m\to\cD^r_n$ are isomorphic as $\mu_r$-torsors. In particular, we see $\sigma$ is induced by the equivariant automorphism $f$ of $D_m$. This proves (\ref{prop:auts-Dnell-reduce-to-id::lift}).

For the remainder of the proof we may therefore assume $g=1$ and $h^r=1+ct^n$. In particular, $h^r=1$ mod $t^n$. Then acting by $h^{-1}\in\mu_r(D_{n-1})=\mu_r(\cD^r_{n-1})$ yields a $2$-isomorphism showing that $\tau\simeq\id$ which proves (\ref{prop:auts-Dnell-reduce-to-id::id}). Lastly, to prove (\ref{prop:auts-Dnell-reduce-to-id::2auts}) note that if $h'$ is any element of $\mu_r(D_{n-1})=\mu_r(\cD^r_{n-1})$, it acts on $f(u)=hu$ by producing the new equivariant automorphism $u\mapsto h'hu$. Thus, if $h'h=h$, we see $h'=1$ as $h$ is a unit.
\end{proof}

We require the following analogue of \cite[Proposition 7.4]{SatrianoUsatine6}.

\begin{proposition}\label{prop:resaut-is-A1-charp}
Let $K$ be a field with
\begin{enumerate}[label=(\alph*)]
\item $\cha(K)=0$ and $n,r\geq2$, or
\item $\cha(K)=p$, $r=p^s r'$ with $\gcd(p,r')=1$, and $n$ is divisible by $p^s$.
\end{enumerate}
Then
\[
\Res_{D_{n,K}/K}\uAut_{D_{n,K}}(\cD^{r}_{n,K})\simeq\bA^1_K.
\]
More specifically, the left hand side is canonically identified with the affine space associated to the $1$-dimensional vector space 
\[
\Ext^0(L_{\cD^{r}_{0,K} / K},\mcJ),
\]
where $\mcJ$ denotes the pull back to $\cD^{r}_{0,K}$ of the ideal sheaf defining the closed immersion $D_{n,K}\hookrightarrow D_{n+1,K}$
\end{proposition}
\begin{proof}
In light of \autoref{prop:auts-Dnell-reduce-to-id}, the proofs of \cite[Corollary 7.10, Corollary 7.11, and Proposition 7.4]{SatrianoUsatine6} all go through the same (as long as $n$ is divisible by $p^s$ in the case where $K$ has positive characteristic) to show
\[
\cJ^{r}_n(\cX)\times_{\sL_n(\sW(\cX))} \Spec L
\]
is the $\bG_a$-torsor, where $\bG_a$ is canonically the vector group associated to the $1$-dimensional vector space
\[
\Ext^0(L_{\cD^{r}_{0,L} / L},\mcJ).
\]
Then \autoref{l:twisted-twisted-discs-def-sp-redone} finishes the proof.
\end{proof}

\begin{proof}[{Proof of \autoref{thm:fibers-twisted->warped-aut-explicit-charp}}]
The proof of \cite[Proposition 7.2]{SatrianoUsatine6}, with \autoref{prop:no-non-trivial-forms-truncated-twisted-charp} used in place of \cite[Proposition 6.1]{SatrianoUsatine6}, holds in our setting and shows 
\[
	\cJ^r_n(\cX)\times_{\sL_n(\sW(\cX))}\Spec L \simeq \Res_{D_{n,L}/L}\uAut_{D_n}(\cD^{r}_{n,L})
\]
where $\Res$ denotes Weil restriction. Thus, the result follows from \autoref{prop:resaut-is-A1-charp}.
\end{proof}

\section{Comparing the motivic measures on warped and twisted arcs}\label{sectionTwistedToWarpedComparison}

The goal of this section is to obtain a generalization in our setting of \cite[Theorem 8.5(a,b)]{SatrianoUsatine6}, which compares the motivic measure on $|\sJ(\cX)|$ to the motivic measure on $|\sL(\widetilde{\sW}(\cX))|$. We begin with the following proposition.

\begin{proposition}\label{propositionTwistedToWarpedFactorsThroughSyntomicLocus}
Let $\cX$ be a smooth finite type Artin stack over $k$ with affine diagonal. Then for all $n \in \Z_{\geq 0}$ and $r \in \Z_{> 0}$, the map $\sJ^r_n(\cX) \to \sL_n(\sW(\cX))$ defined in \autoref{sectionTwistedToWarpedFibers} factors through $\sL_n(\widetilde{\sW}(\cX))$.
\end{proposition}

\begin{proof}
By \autoref{l:fppf-loc-sm->sm}, we have that $\cD^r_k$ is regular, so $\cD^r_k \to D_k$ is syntomic. Therefore for any $k$-algebra $A$, we have $\cD^r_{0, A} \to D_{0,A}$ is syntomic. Therefore $\sJ_0^r(\cX) \to \sL_0(\sW(\cX)) = \sW(\cX)$ factors through $\sW^\synt(\cX)$. Thus in order to show that $\sJ_0^r(\cX) \to \sW(\cX)$ factors through $\widetilde{\sW}(\cX)$, it is sufficient to show that for any field extension $k'$ of $k$ and any warped map of the form $(\cD^r_{0, k'} \to \Spec(k'), \cD^r_{0, k'} \to \cX)$, the corresponding point $\Spec(k') \to \sW(\cX)$ is the special point of a map $D_{k'} \to \sW(\cX)$ whose generic point lands in $\cX$. Since $\cD_{k'} \to D_{k'}$ is an isomorphism over the generic point $\Spec(k'\llparenthesis t \rrparenthesis) \hookrightarrow D_{k'}$, it is thus sufficient to show that there exists a map $\cD^r_{k'} \to \cX$ whose composition with the closed immersion $\cD^r_{0, k'} \hookrightarrow \cD^r_{k'}$ is (2-isomorphic to) the given map $\cD^r_{0, k'} \to \cX$. Such a map $\cD^r_{k'} \to \cX$ exists by infinitesimal lifting for cohomologically affine stacks \cite[Proposition 2.5]{SatrianoUsatine6}, whose proof in loc. cit. never uses that $k$ is algebraically closed or characteristic 0, and \autoref{remarkTwistedArcsAreTwistedArcs}. Therefore the map $\sJ^r_0(\cX) \to \sW(\cX)$ factors through $\widetilde{\sW}(\cX)$, i.e., the proposition holds when $n = 0$. The full proposition then follows from the fact that $\sL_n(\widetilde{\sW}(\cX)) = (\theta^n_0)^{-1}(\sL_0(\widetilde{\sW}(\cX)))$.
\end{proof}

In the theorem below, the map $\sJ^r(\cX) \to \sL(\widetilde{\sW}(\cX))$ is the map given by taking an inverse limit over $n$ of the maps $\sJ^r_n(\cX) \to \sL_n(\widetilde{\sW}(\cX))$ obtained by the conclusion of \autoref{propositionTwistedToWarpedFactorsThroughSyntomicLocus}.

\begin{theorem}\label{theoremComparingWarpedToTwisted}
Let $\cX$ be a smooth finite type equidimensional Artin stack over $k$ with affine diagonal and a good moduli space. Let $r \in \Z_{>0}$, let $\cC \subset |\sJ^r(\cX)|$, and let $\cE \subset |\sL(\widetilde{\sW}(\cX)|$ be the image of $\cC$ along the map $|\sJ^r(\cX)| \to |\sL(\widetilde{\sW}(\cX)|$.

\begin{enumerate}[label=(\alph*)]

\item If $k'$ is a field extension of $k$, then the map $\overline{\cC}(k') \to \overline{\cE}(k')$ is bijective.

\item\label{itemTwoOfTwistedToWarpedTheorem} $\cC$ is a cylinder if and only if $\cE$ is a cylinder if and only if $\cE$ is a bounded cylinder. In that case,

\[
	\nu_\cX(\cC) = \begin{cases} \mu_{\widetilde{\sW}(\cX)}(\cE), \quad &r = 1 \\ \bL\mu_{\widetilde{\sW}(\cX)}(\cE), &\text{otherwise} \end{cases}.
\]

\end{enumerate}
\end{theorem}

\begin{proof}
This is \cite[Theorem 8.5(a,b)]{SatrianoUsatine6} without the assumption that $k$ is algebraically closed and characteristic 0. The proof in loc. cit. works verbatim with the following adaptations. \autoref{FMN-tame-stacks} must be used in place of \cite[Proposition A.1]{FantechiMannNironi}. \autoref{prop:no-non-trivial-forms-truncated-twisted-charp} must be used in place of \cite[Proposition 6.1]{SatrianoUsatine6}. In the proof of \cite[Lemma 6.5]{SatrianoUsatine6}, we must use \autoref{prop:twisted-twisted-discs-defs-truncated-twisted-charp} for the cases not handled by \cite[Proposition 6.3]{SatrianoUsatine6}. \autoref{thm:fibers-twisted->warped-aut-explicit-charp} must be used in place of \cite[Theorem 7.1]{SatrianoUsatine6}. Finally, if $k$ has characteristic $p > 0$, in the paragraph where loc. cit. is used, choose $n, \cE_n$ such that $n$ is divisible by $p^s$, where $r = p^s r'$ with $\gcd(p, r') = 1$.
\end{proof}

We end this section by stating a corollary that will later be used to show that the motivic measure for measurable sets of twisted arcs is well defined.

\begin{corollary}\label{lemmaTwistedBoundedCylinderFiniteSubcover}
Let $\cX$ be a smooth finite type equidimensional Artin stack over $k$ with affine diagonal and a good moduli space, and let $\cC \subset |\sJ(\cX)|$ be a bounded cylinder. Then any cover of $\cC$ by bounded cylinders has a finite subcover.
\end{corollary}

\begin{proof}
This follows immediately from the first part of \autoref{theoremComparingWarpedToTwisted}\ref{itemTwoOfTwistedToWarpedTheorem} and from the untwisted version, \autoref{lemmaUntwistedBoundedCylinderFiniteSubcover}, applied to $\widetilde{\sW}(\cX)$.
\end{proof}

\section{Change of variables and thin subsets for twisted arcs}\label{sectionTwistedChangeOfVariables}

We now obtain the change of variables formula for twisted arcs.

\begin{theorem}\label{theoremChangeOfVariablesForTwistedArcs}
Let $\cX$ be a smooth irreducible finite type Artin stack over $k$ with affine diagonal and a good moduli space, let $Y$ be a finite type irreducible scheme over $k$, and let $\cX \to Y$ be a morphism. Let $\cU$ be an open substack of $\cX$ such that the composition $\cU \hookrightarrow \cX \to Y$ is an open immersion, let $\cC \subset |\sJ(\cX)|$ be a bounded cylinder that is disjoint from $|\sJ(\cX \setminus \cU)|$, and let $D \subset \sL(Y)$ be a cylinder. Assume that for every field extension $k'$ of $k$, the map $\overline{\cC}(k') \to D(k')$ induced by $\cX \to Y$ is a bijection.
\begin{enumerate}[label=(\alph*)]

\item\label{theoremChangeOfVariablesIntegerPart} The restriction of $\het_{\cX/Y} + \wt_\cX$ to $\cC$ is integer valued and takes only finitely many values, and the set $(\het_{\cX/Y} + \wt_\cX)^{-1}(n) \cap \cC$ is a bounded cylinder for all $n \in \Z$.

\item We have the equality
\[
	\mu_Y(D) = \int_{\cC} \bL^{-\het_{\cX/Y} - \wt_\cX}\diff\nu_\cX.
\]

\end{enumerate}
\end{theorem}

\begin{proof}
This is \cite[Theorem 8.5(a,b)]{SatrianoUsatine6} without the assumption that $k$ is algebraically closed and characteristic 0, but with the additional assumption that $\cX$ is irreducible. The proof in loc. cit. works verbatim if we use \autoref{theoremComparingWarpedToTwisted} in place of \cite[Theorem 8.5]{SatrianoUsatine6}, \autoref{theoremUntwistedChangeOfVariables} in place of \cite[Theorem 7.1]{SatrianoUsatine5}, and \autoref{lemmaExtComputationTwistedDiscOverDisc} in place of \cite[Lemma 11.3]{SatrianoUsatine6} in the proof of \cite[Theorem 11.2]{SatrianoUsatine6}.
\end{proof}

The remainder of this section will be used to define measurable subsets of twisted arcs, obtain the twisted version of \autoref{theoremUntwistedThinSubsets}, and then use these in applying \autoref{theoremChangeOfVariablesForTwistedArcs} to tame crepant resolutions.

\begin{definition}
Let $\cX$ be a smooth finite type equidimensional Artin stack over $k$ with affine diagonal and a good moduli space, let $\cC \subset |\sJ(\cX)|$, and let $\varepsilon \in \R_{>0}$. A \emph{bounded cylinderical $\varepsilon$-approximation} of $\cC$ is a pair $(\cC^{(0)}, \{\cC^{(i)}\}_{i \in I})$ such that
\begin{itemize}

\item $\cC^{(0)} \subset |\sJ(\cX)|$ is a bounded cylinder,

\item $\{\cC^{(i)}\}_{i \in I}$ is a collection with each $\cC^{(i)} \subset |\sJ(\cX)|$ a bounded cylinder,

\item $\Vert \nu_\cX(\cC^{(i)}) \Vert < \varepsilon$ for all $i \in I$, and

\item $(\cC \cup \cC^{(0)}) \setminus (\cC \cap \cC^{(0)}) \subset \bigcup_{i \in I} \cC^{(i)}$.

\end{itemize}
\end{definition}

\begin{definition}
Let $\cX$ be a smooth finite type equidimensional Artin stack over $k$ with affine diagonal and a good moduli space, and let $\cC \subset |\sJ(\cX)|$. The set $\cC$ is called \emph{measurable} if it has a bounded cylindrical $\varepsilon$-approximation for all $\varepsilon \in \R_{>0}$.
\end{definition}

\begin{proposition}\label{propositionTwistedMeasurableSetHasWellDefinedVolume}
Let $\cX$ be a smooth finite type equidimensional Artin stack over $k$ with affine diagonal and a good moduli space, and let $\cC \subset |\sJ(\cX)|$ be a measurable set. There exists a unique element $\nu_{\cX}(\cC) \in \widehat{\sM}_k$ satisfying $\Vert \nu_\cX(\cC) - \nu_\cX(\cC^{(0)}) \Vert < \varepsilon$ for any bounded cylindrical $\varepsilon$-approximation $(\cC^{(0)}, \{\cC^{(i)}\}_{i \in I})$.
\end{proposition}

\begin{proof}
The proof is exactly as for \autoref{propositionUntwistedMeasurableSetHasWellDefinedVolume} except we use \autoref{lemmaTwistedBoundedCylinderFiniteSubcover} in place of \autoref{lemmaUntwistedBoundedCylinderFiniteSubcover}.
\end{proof}

\begin{definition}
Let $\cX$ be a smooth finite type equidimensional Artin stack over $k$ with affine diagonal and a good moduli space, and let $\cC \subset |\sJ(\cX)|$ be a measurable set. The \emph{motivic volume} of $\cC$ is the element $\nu_{\cX}(\cC)$ in the statement of \autoref{propositionTwistedMeasurableSetHasWellDefinedVolume}.
\end{definition}

\begin{remark}\label{remarkTwistedToWarpedForMeasurable}
Using \autoref{theoremComparingWarpedToTwisted}, it is straightforward to check that under the hypotheses of \autoref{theoremComparingWarpedToTwisted}, $\cC$ is measurable if and only if $\cE$ is measurable, and in that case,
\[
	\nu_\cX(\cC) = \begin{cases} \mu_{\widetilde{\sW}(\cX)}(\cE), \quad &r = 1 \\ \bL\mu_{\widetilde{\sW}(\cX)}(\cE), &\text{otherwise} \end{cases}.
\]
\end{remark}

\begin{theorem}\label{theoremTwistedThinSubsets}
Let $\cX$ be a smooth finite type equidimensional Artin stack over $k$ with affine diagonal and a good moduli space, and let $\cZ$ be a closed substack of $\cX$ with $\dim\cZ < \dim\cX$. For any $r \in \Z_{>0}$, the set $|\sJ^r(\cZ)|$ is a measurable subset of $|\sJ^r(\cX)|$ with $\nu_\cX(|\sJ^r(\cZ)|) = 0$.
\end{theorem}

\begin{proof}
This is \cite[Theorem 9.1]{SatrianoUsatine6} without the assumption that $k$ is algebraically closed and characteristic 0. The proof in loc. cit. works as long as we use \autoref{remarkTwistedToWarpedForMeasurable} in place of \cite[Theorem 8.5]{SatrianoUsatine6} and \autoref{theoremUntwistedThinSubsets} in place of \cite[Theorem 9.2]{SatrianoUsatine6}.
\end{proof}

We end this section with the following application of \autoref{theoremChangeOfVariablesForTwistedArcs} and \autoref{theoremTwistedThinSubsets} to tame crepant resolutions of singularities.

\begin{corollary}\label{corollaryMotivicIntegralCrepantResolution}
Let $\cX$ be a smooth finite type irreducible Artin stack over $k$, and let $Y$ be a $\Q$-Gorenstein finite type irreducible scheme over $k$. If $\cX \to Y$ is a tame proper birational map with $K_{\cX/Y} = 0$, then $\bL^{\ord_Y^\Gor}$ is integrable on $\sL(Y)$. Furthermore if $\cX$ has affine diagonal, then
\[
	\mu^\Gor_Y(\sL(Y)) = \int_{\sJ(\cX)} \bL^{-\wt_\cX} \diff\nu_\cX.
\]
\end{corollary}

\begin{remark}
In the statement of \autoref{corollaryMotivicIntegralCrepantResolution}, $\cX$ is tame over the scheme $Y$ and is thus tame over $k$. Thus $\cX$ has a good moduli space. Since $\cX$ is tame over $k$, $I_{\mu_r}(\cX)$ is empty for all but finitely many $r$, so $\sJ^r(\cX)$ is empty for all but finitely many $r$. Altogether, along with \autoref{propositionWeightFunctionLocallyConstant}, the integral $\int_{\sJ(\cX)} \bL^{-\wt_\cX} \diff\nu_\cX$ is well defined in the statement of \autoref{corollaryMotivicIntegralCrepantResolution}.
\end{remark}

\begin{proof}
Since $\bL^{\ord_Y^\Gor}$ being integrable can be checked on an open cover of $Y$ and the hypotheses imply $\cX \to Y$ has affine diagonal, we may assume the $\cX$ has affine diagonal.

Now let $V$ be a nonempty open subscheme of $Y$ such that $\cX \to Y$ is an isomorphism over $V$, and set $\cU = \cX \times_Y V$. By \cite{BrescianiVistoli}, the map 
\[
	\overline{(|\sJ(\cX)| \setminus |\sJ(\cX \setminus \cU)|)}(k') \to (\sL(Y) \setminus \sL(Y \setminus V))(k')
\]
is bijective for any field extension $k'$ of $k$. In particular for any cylinder $D \subset \sL(Y)$ disjoint from $\sL(Y \setminus V)$, if $\cC$ is the preimage of $D$ in $|\sJ(\cX)|$, then $\overline{\cC}(k') \to D(k')$ is bijective. Now for every $n \in \Z_{\geq 0}$, set $D^{(n)} = (\theta_n)^{-1}(\sL_n(Y) \setminus \sL_n(Y \setminus V)) \subset \sL(Y)$, where we give $Y \setminus V$ the reduced closed subscheme structure, and let $\cC^{(n)}$ be the preimage of $D^{(n)}$ in $|\sJ(\cX)|$. Since the ideal defining $\ord^\Gor_Y$ is supported away from the smooth scheme $V$, the function $\ord^\Gor_Y$ does not take the value $\infty$ on $D^{(n)} \subset \sL(Y) \setminus \sL(Y \setminus V)$. Thus $D^{(n)}$ is covered by finitely many cylinder of the form $D^{(n)} \cap (\ord^\Gor_Y)^{-1}(m)$ (see e.g., \autoref{lemmaUntwistedBoundedCylinderFiniteSubcover}). Thus $\bL^{\ord^\Gor_Y}$ is integrable on $D^{(n)}$, and it is straightforward to use \autoref{theoremChangeOfVariablesForTwistedArcs} to show that
\[
	\int_{D^{(n)}} \bL^{\ord^\Gor_Y}\diff\mu_Y = \int_{\cC^{(n)}} \bL^{\ord^\Gor_Y \circ f -\het_{\cX/Y} - \wt_\cX}\diff\nu_\cX,
\]
where $f: |\sJ(\cX)| \to \sL(Y)$ is the map induced by $\cX \to Y$. Then by \cite[Proposition 12.2]{SatrianoUsatine6}, whose proof never uses that $k$ is algebraically closed or characteristic 0,
\[
	\int_{D^{(n)}} \bL^{\ord^\Gor_Y}\diff\mu_Y = \int_{\cC^{(n)}} \bL^{- \wt_\cX}\diff\nu_\cX.
\]
Since $\sL(Y \setminus V)$ is measurable in $\sL(Y)$ and $\mu_Y(\sL(Y \setminus V)) = 0$, we have $\bL^{\ord^\Gor_Y}$ is integrable on $\sL(Y)$ if and only if $\bL^{\ord^\Gor_Y}$ is integrable on $\sL(Y) \setminus \sL(Y \setminus V)$, and in that case
\[
	\int_{\sL(Y)} \bL^{\ord^\Gor_Y} \diff\mu_Y = \int_{\sL(Y) \setminus \sL(Y \setminus V)} \bL^{\ord^\Gor_Y} \diff\mu_Y.
\]
Using also that $D^{(n)} \subset D^{(n+1)}$ and $\sL(Y) \setminus \sL(Y \setminus V) = \bigcup_{n \in \Z_{\geq 0}} D^{(n)}$, we see that $\bL^{\ord^\Gor_Y}$ is integrable on $\sL(Y)$ if and only if the sequence $\{\int_{D^{(n)}} \bL^{\ord^\Gor_Y}\diff\mu_Y\}_n$ converges, and in that case
\[
	\int_{\sL(Y)} \bL^{\ord^\Gor_Y} \diff\mu_Y = \lim_{n \to \infty} \int_{D^{(n)}} \bL^{\ord^\Gor_Y}\diff\mu_Y.
\]
Therefore $\bL^{\ord^\Gor_Y}$ is integrable on $\sL(Y)$ if and only if the sequence $\{\int_{\cC^{(n)}} \bL^{- \wt_\cX}\diff\nu_\cX\}_n$ converges, and in that case
\[
	\int_{\sL(Y)} \bL^{\ord^\Gor_Y} \diff\mu_Y = \lim_{n \to \infty} \int_{\cC^{(n)}} \bL^{- \wt_\cX}\diff\nu_\cX.
\]
By \autoref{theoremTwistedThinSubsets}, $|\sJ(\cX \setminus \cU)|$ is a measurable subset of $|\sJ(\cX)|$ with measure 0. Thus using that $\cC^{(n)} \subset \cC^{(n+1)}$ and $\bigcup_{n \in \Z_{\geq 0}} \cC^{(n)} = |\sJ(\cX)| \setminus |\sJ(\cX \setminus \cU)|$, the sequence $\{\int_{\cC^{(n)}} \bL^{- \wt_\cX}\diff\nu_\cX\}_n$ converges and
\[
	\lim_{n \to \infty} \int_{D^{(n)}} \bL^{\ord^\Gor_Y}\diff\mu_Y = \int_{\sJ(\cX)} \bL^{-\wt_\cX} \diff\nu_\cX,
\]
completing our proof.
\end{proof}

\section{Twisted jets and cyclotomic inertia}

The goal of this section is to prove \autoref{maintheoremGorensteinMeasureCrepantResolution} as a consequence of \autoref{corollaryMotivicIntegralCrepantResolution}. We will spend most of this section proving the following intermediate result, which is a motivic analog of \cite[Proposition 9.1]{HuangSatrianoUsatine} but in a more general setting. In some ways, the proof follows the structure of loc. cit., but we have the added complication of working with two term complexes since $L_\cX$ is not necessarily a single vector bundle. In positive characteristic, this occurs even when $\cX$ is tame.

\begin{proposition}\label{propositionTwistedJetsToInertiaComponents}
Let $\cX$ be a smooth equidimensional finite type Artin stack over $k$ with affine diagonal, and let $\cY$ be a connected component of $I_\mu(\cX)$. Then
\[
	\e(\sJ_0(\cX) \times_{I_\mu(\cX)} \cY) = \bL^{\dim\cX - \dim\cY} \e(\cY).
\]
\end{proposition}

Before we prove \autoref{propositionTwistedJetsToInertiaComponents}, we will show how \autoref{maintheoremGorensteinMeasureCrepantResolution} follows from \autoref{corollaryMotivicIntegralCrepantResolution} and \autoref{propositionTwistedJetsToInertiaComponents}.

\begin{proof}[Proof of \autoref{maintheoremGorensteinMeasureCrepantResolution}]
Since everything we need to prove can be checked Zariski locally on $Y$, we may assume that $\cX$ has affine diagonal. Then by \autoref{corollaryMotivicIntegralCrepantResolution}, we have $\bL^{\ord^\Gor_Y}$ is integrable on $\sL(Y)$ and
\[
	\mu^\Gor_Y(\sL(Y)) = \int_{\sJ(\cX)} \bL^{-\wt_\cX} \diff\nu_\cX = \bL^{-\dim\cX} \sum_{\cY} \bL^{-\overline{\wt}_\cX(\cY)} \e(\sJ_0(\cX) \times_{I_\mu(\cX)} \cY),
\]
where the second equality is by the definition of $\nu_\cX$, the definition of $\wt_\cX$, and \autoref{propositionWeightFunctionLocallyConstant}, and the sum is over all connected components $\cY$ of $I_\mu(\cX)$. Everything but the last sentence of the desired result then follows from \autoref{propositionTwistedJetsToInertiaComponents}, the definition of $\shft_\cX(\cY)$ in terms of $\overline{\wt}_\cX(\cY)$, and the fact that $\dim \cX = \dim Y$.

Now assume that $Y$ is $1$-Gorenstein. By the definitions, it is sufficient to show that $\wt_\cX$ is integer valued. Let $U$ be a nonempty open subscheme of $Y$ over which $\cX \to Y$ is an isomorphism, and let $\cZ$ be a closed substack of $\cX$ supported on the complement of the preimage of $U$. By \autoref{propositionWeightFunctionLocallyConstant}, it is sufficient to show that $\wt_\cX$ takes integer values outside a set of measure 0, so by \autoref{theoremTwistedThinSubsets}, it is sufficient to show that $\wt_\cX$ is integer valued on $|\sJ(\cX)| \setminus \theta_n^{-1}(|\sJ_n(\cZ)|)$ for all $n$. Then by \autoref{theoremChangeOfVariablesForTwistedArcs}\ref{theoremChangeOfVariablesIntegerPart}, it is sufficient to show that $\het_{\cX/Y}$ is integer valued on $|\sJ(\cX)| \setminus \theta_n^{-1}(|\sJ_n(\cZ)|)$. Using that $Y$ is 1-Gorenstein, the latter holds by \cite[Proposition 12.2]{SatrianoUsatine6}, whose proof in loc. cit. never uses that $k$ is algebraically closed or characteristic 0.
\end{proof}

The next subsection is devoted to proving \autoref{propositionTwistedJetsToInertiaComponents}.

\subsection{Fibers from twisted jets to cyclotomic inertia}

We will fix some notation that will be used throughout this subsection. Let $\cX$ be a smooth equidimensional finite type Artin stack over $k$ with affine diagonal. Given $\gamma\colon\Spec k'\to I_{\mu_r}(\cX)$ for a field extension $k'$ of $k$, let $F_\gamma\colon B\mu_r\to\cX$ denote the associated representable map. Let
\[
L^iF_\gamma^*L_\cX=\bigoplus_{a=1}^r\cO(a)^{\oplus d^i_a(\gamma)}.
\]
For any $b\in\bZ$, we define
\[
d^i_b(\gamma):=d^i_c(\gamma)\quad\textrm{such\ that}\quad c\in(0,r]\quad\textrm{and}\quad b\equiv c\pmod{r}.
\]
If $r, n \in \Z_{>0}$, and $A$ is a $k$-algebra, we let
\[
	\cC^r_{n,A} = [\Spec( A[t^{1/r}]/(t^{n/r}) ) / \mu_r],
\]
where $\mu_r$ acts with weight $1$ on $t^{1/r}$. %Then $\cC^r_{1,A} = B\mu_r \otimes_k A$ and $\cC^r_{r, A} = \cD^r_{0,A}$. For each $r, n \in \Z_{>0}$, we will set
Let
\[
	\cZ^r_n = \uHom^{\rep}_k(\cC^r_{n,k}, \cX),
\]
so for example, $\cZ^r_1 = I_{\mu_r}(\cX)$ and $\cZ^r_r = \sJ^r_0(\cX)$. For any component $\cY$ of $I_{\mu_r}(\cX)$, let
\[
	\cY^r_n:=\cZ^r_n\times_{I_{\mu_r}(\cX)}\cY.
\]

We now generalize \cite[Proposition 9.4]{HuangSatrianoUsatine}.

\begin{proposition}\label{propSubTruncationFiberComputation-charp}
Let $r, n \in \Z_{>0}$, assume $n+1 \leq r$, let $\cY$ be a connected component of $I_{\mu_r}(\cX)$, let $k'$ be a field extension of $k$, and let $\Spec(k') \to \cY^r_n$ be a morphism, and let $\gamma\colon\Spec k'\to I_{\mu_r}(\cX)$ be the induced map. Then
\[
	\cY^r_{n+1} \times_{\cY^r_n} \Spec(k') \cong \bA^{d^0_{-n}(\gamma)}_{k'}\times B\bG_{a,k'}^{d^1_{-n}(\gamma)}.
\]
as stacks over $k'$. Furthermore, 
\[
d^0_{-n}(\gamma)-d^1_{-n}(\gamma)
\]
depends only on $\cY$.
\end{proposition}
\begin{proof}
As in the proof of \cite[Proposition 4.11]{SatrianoUsatine6}, it suffices to compute 
\[
	\dim\Ext^i(LF_\gamma^*L_\cX,\cI_n),
\]
where $\cI_n$ is the pullback of the ideal sheaf cutting out $\cY^r_n$ in $\cY^r_{n+1}$. As in \cite[Proposition 9.4]{HuangSatrianoUsatine}, $\cI_n\simeq\cO(-n)$. By \cite[Tag 07AA]{stacks-project}, we have a spectral sequence
\[
E_2^{ij}=\Ext^i(L^{-j}F_\gamma^*L_\cX,\cO(-n))\Rightarrow \Ext^{i+j}(LF_\gamma^*L_\cX,\cO(-n)).
\]
All terms vanish except when $j\in[-1,0]$. Since each $L^{-j}F_\gamma^*L_\cX$ is locally free,
\[
\Ext^i(L^{-j}F_\gamma^*L_\cX,\cO(-n))=H^i(L^{-j}F_\gamma^*L_\cX(-n)).
\]
Since $B\mu_r$ is cohomologically affine, this term vanishes unless $i=0$. Thus,
\[
\dim\Ext^j(LF_\gamma^*L_\cX,\cO(-n))=\sum_a \dim H^0(\cO(-n-a))^{\oplus d^{-j}_a(\gamma)}.
\]
This quantity vanishes unless $-n-a\equiv 0\pmod{r}$, so we see $a\in(0,r]$ is congruent to $-n$, which proves
\[
	\cY^r_{n+1} \times_{\cY^r_n} \Spec(k') \cong \bA^{d^0_{-n}(\gamma)}_{k'}\times B\bG_{a,k'}^{d^1_{-n}(\gamma)}.
\]
Lastly, the fact that
\[
d^0_{-n}(\gamma)-d^1_{-n}(\gamma)
\]
depends only on $\cY$ follows from the proof of \cite[Proposition 4.19]{SatrianoUsatine6}. Indeed the statement of (loc.~cit.) says $\overline{\wt}_\cX$ is locally constant on $|I_{\mu_r}(\cX)|$ however the proof shows that for each $w$, the difference $d^0_{-w}(\gamma)-d^1_{-w}(\gamma)$ is locally constant; see the 7th, 9th, and 10th displayed equations of the proof.
\end{proof}

We may now prove \autoref{propositionTwistedJetsToInertiaComponents}.

\begin{proof}[Proof of \autoref{propositionTwistedJetsToInertiaComponents}]
By \autoref{propSubTruncationFiberComputation-charp} and \cite[Proposition 2.4]{SatrianoUsatine6}, whose proof in loc. cit. never uses that $k$ is algebraically closed or has characteristic 0, there exists some $d \in \Z$ such that
\[
	\e(\sJ_0(\cX) \times_{I_\mu(\cX)} \cY) = \bL^{d} \e(\cY).
\]
Then by \autoref{propositionDimensionAndGrothendieckNorm}, $d = \dim(\sJ_0(\cX) \times_{I_\mu(\cX)} \cY) - \dim\cY$. By \cite[Proposition 5.4]{HuangSatrianoUsatine}, whose proof in loc. cit. never uses that $k$ is algebraically closed or has characteristic 0, we have $\sJ_0(\cX)$ is equidimensional and has dimension $\dim\cX$. The result then follows from the fact that $\sJ_0(\cX) \times_{I_\mu(\cX)} \cY$ is a nonempty open substack of $\sJ_0(\cX)$.
\end{proof}

\section{Quotients by finite linearly reductive subgroups of special linear groups}

We would like to apply \autoref{maintheoremGorensteinMeasureCrepantResolution} to the map $[V/G] \to V/G$, where $G$ is a finite linearly reductive subgroup scheme of $\SL(V)$ for a finite dimensional $k$-vector space $V$. The purpose of this section is to prove the following, which shows that $[V/G] \to V/G$ satisfies the hypotheses of \autoref{maintheoremGorensteinMeasureCrepantResolution}.

\begin{theorem}\label{theoremQuotientIsQGorenstein}
Let $V$ be a finite dimensional vector space over $k$, and let $G$ be a finite linearly reductive subgroup scheme of $\SL(V)$. Then $V/G$ is Gorenstein, the map $[V/G] \to V/G$ is birational, and $K_{[V/G]/(V/G)} = 0$.
\end{theorem}

\subsection{Freeness in codimension $1$}

We begin by showing that finite diagonalizable subgroup schemes of $\SL(V)$ must act freely in codimension $1$ on $V$, after which we turn to the general case.

\begin{lemma}\label{l:generic-rep-diagonalizable}
Let $G\subset\SL(V)$ be a faithful representation of a finite diagonalizable group scheme over an algebraically closed field $k$. Then $G$ is a generic representation, i.e., it acts freely in codimension $1$.
\end{lemma}
\begin{proof}
Suppose there is a codimension $1$ point $v$ with non-trivial stabilizer $H$. Since all subgroups of $G$ are diagonalizable, we see $H$ is diagonalizable. Furthermore, every non-trivial diagonalizable group contains a copy of $\mu_q$ for some prime $q$. Thus, we have reduced to the case where $G=\mu_q$.

We may choose coordinates $V\simeq\bA^n$ such that $\mu_q$ acts diagonally with weights $(w_1,\dots,w_n)$. Then a codimension $1$ locus has non-trivial stabilizer if and only if it is a coordinate hyperplane $x_j=0$ with the property that all $w_i=0$ for $i\neq j$. But $\mu_q\subset\SL(V)$, so $\sum_{m=1}^n w_i=0$ which implies $w_j=0$. This contradicts faithfulness of the representation.
\end{proof}

\begin{proposition}\label{prop:generic-rep}
Let $G\subset\SL(V)$ be a faithful representation of a finite linearly reductive group scheme over an algebraically closed field $k$. Then $G$ is a generic representation, i.e., it acts freely in codimension $1$.
\end{proposition}
\begin{proof}
Let $W\subset V$ be hyperplane and assume there exists a $k$-scheme $S$ and $g\in G(S)$ with $gw=w$ for all $w\in W(S)$. We show $g=1$.

Choosing an affine cover of $S$, it suffices to assume $S=\Spec R$. Choose a basis $v_1,\dots,v_n$ of $V$ over $k$ with the property that $v_1,\dots,v_{n-1}$ is a basis for $W$. Then $v_1,\dots,v_n$ is an $R$-basis for $V(R)$ and  $v_1,\dots,v_{n-1}$ is an $R$-basis for $W(R)$. We see $gv_i=v_i$ for $i<n$ and since $g\in\SL_n(R)$, we must have $gv_n=v_n+\sum_{i<n}r_iv_i$. Letting $m$ be the order of $g$, we find $g^mv_n=v_n+\sum_{i<n}mr_iv_i$, hence $mr_i=0$. If $\cha(k)=0$, then we see $r_i=0$ so $g=1$. If $\cha(k)=p$ and $g\neq1$, then some $r_i\neq0$ and so $p$ divides $m$.

By the proof of \cite[Lemma 2.17]{AOV}, we know $G=\Delta\rtimes H$ with $\Delta$ diagonalizable and the connected component of the identity, and $H$ a finite \'etale tame constant group scheme. Then the image of $g^{|H|}$ in $H$ is trivial, so $g^{|H|}\in\Delta(R)$. Furthermore, $g^{|H|}\neq1$ since $p$ divides the order of $g$ and $|H|$ is prime to $p$. Therefore, replacing $g$ by $g^{|H|}$, we may assume $g\in\Delta(R)$, $g\neq1$, and $gw=w$ for all $w\in W(R)$. The result, therefore follows from \autoref{l:generic-rep-diagonalizable}.
\end{proof}

\subsection{The cotangent complex of $[V/G]$ and the weight function of $BG$}

To prove $\omega_{V/G}$ is torsion, we make use of the determinant of the cotangent complex $L_{[V/G]}$, which we calculate in \autoref{prop:LVmodG}.

\begin{proposition}\label{prop:LBG}
If $G$ is a finite linearly reductive group scheme over an algebraically closed field $k$, then 
\[
\det L_{BG/k}\simeq\cO_{BG}\quad\textrm{and}\quad \overline{\wt}_{BG}=0.
\]
\end{proposition}
\begin{proof}
By the proof of \cite[Lemma 2.17]{AOV}, we know $G=\Delta\rtimes H$ where $\Delta$ is the connected component of the identity which is diagonalizable, and $H$ is a finite \'etale tame constant group scheme. Write $\Delta=\Spec k[A]$ with $A$ a finite abelian group. We have a map $H\to\Aut(\Delta)=\Aut(A)$, so choosing a surjection $M\to A$ in the category of $k[H]$-modules from a finite rank free abelian group $M$, we obtain an $H$-equivariant embedding
\[
\Delta\hookrightarrow T:=\Spec k[M]
\]
into a torus $T$. Via the $H$-action on $T$, we may form the group scheme $G'=T\rtimes H$. Then there is a natural $G'$-action on $T'=T/\Delta=\Spec k[M']$ where $T$ acts by left multiplication and $H$ acts through its action on $T$. It is straightforward to see that this $G'$-action is transitive and that the stabilizer of the identity is $G$. Thus,
\[
BG\simeq[T'/G'].
\]
From this presentation, and letting $\rho\colon T'\to BG$ be the smooth cover, we have an exact triangle
\[
\rho^*L_{BG}\to \Omega^1_{T'}\to \cO_{T'}\otimes(\g')^\vee
\]
where $\g'$ is the Lie algebra of $G'$ (which is the same as the Lie algebra of $T$). Note that 
\[
\Omega^1_{T'}\simeq\cO_{T'}\otimes(M'\otimes k)\quad\textrm{and}\quad \g'\simeq M\otimes k,
\]
so 
\[
\rho^*\det L_{BG}=\omega_{T'}\otimes\det\g' = \cO_{T'}\otimes(\det(M'\otimes k)\otimes_k\det(M\otimes k)^\vee).
\]
Since $T$ acts trivially on $M'\otimes k=\Omega_{T',1}$ and on its Lie algebra $M\otimes k$, the $G'$-action on $\det(M'\otimes k)\otimes\det(M\otimes k)^\vee$ factors through the $H$-action.

To finish the proof, it therefore suffices to show that $M'\otimes k$ and $M\otimes k$ are isomorphic as $H$-representations. Indeed, it is clear how this implies $\det L_{BG}\simeq\cO_{BG}$. To see how it implies $\overline{\wt}_{BG}=0$, recall that if $B\mu_r\to BG$ is an representable map then $\overline{\wt}_{BG}$ is defined (note that $\dim BG = 0$) as a weighted sum of the differences
\[
\dim(M'\otimes k)_w-\dim(M'\otimes k)_w
%\wt_{BG}=\sum_{w=1}^r w(\dim(M'\otimes k)_w-\dim(M'\otimes k)_w)
\]
where the subscript $w$ indicates the graded piece indexed by $w$. Note also that $(M\otimes k)_w=M_w\otimes k$ and similarly for $M'$.

We now show that $M'\otimes k$ and $M\otimes k$ are isomorphic as $H$-representations. To begin with, since $k$ is algebraically closed, $H$ is the base change to $k$ of a constant group defined over $\bZ$; we will not distinguish $H$ and the constant group defined over $\bZ$. We note that the quotient map $T\to T'=T/\Delta$ induces a finite index $H$-equivariant inclusion $M'\subset M$. This induces an $H$-equivariant isomorphism $M'\otimes\QQ\simeq M\otimes\QQ$. So for all $h\in H$, we see $\Tr(h| M'\otimes\QQ)=\Tr(h| M\otimes\QQ)$, and hence $\Tr(h| M')=\Tr(h| M)$. Thus, $\Tr(h| M'\otimes k)=\Tr(h| M\otimes k)$, i.e., $M'\otimes k$ and $M\otimes k$ have the same character, which implies $M'\otimes k\simeq M\otimes k$ as $H$-representations since $H$ is constant and linearly reductive.
\end{proof}

\begin{proposition}\label{prop:LVmodG}
Keep the notation of \autoref{prop:LBG} and let $G\subset\SL(V)$. Then 
\[
\det L_{[V/G]}\simeq\cO_{[V/G]}
\]
\end{proposition}
\begin{proof}
Let $\cX=[V/G]$ and consider the cartesian square
\[
\xymatrix{
V\ar[d]_-{\rho}\ar[r]^-{f'} & \Spec k\ar[d]\\
\cX\ar[r]^-{f} & BG
}
\]
We have an exact triangle
\[
L_\cX\to Lf^*L_{BG}\to L_f,
\]
from which we see
\[
\det L_\cX\simeq f^*\det(L_{BG})\otimes \det(L_f)^\vee\simeq \det(L_f)^\vee,
\]
where the second isomorphism uses \autoref{prop:LBG}. Thus, it remains to show $\det(L_f)$ is trivial, or equivalently, $\rho^*\det(L_f)=\det\Omega^1_V$ is trivial as a $G$-representation. This is true since $G\subset\SL(V)$.
\end{proof}

\subsection{Proving the canonical bundle on $V/G$ is trivial}

\begin{proposition}\label{prop:Q-Gorenstein}
Let $V$ be a representation of $G\subset\SL(V)$, a finite linearly reductive group scheme over an algebraically closed field $k$. Then $\omega_{V/G}\simeq\cO_{V/G}$, so $V/G$ is $1$-Gorenstein. 

Furthemore, there exists an open subset $U \subset V/G$ such that $[V/G] \to V/G$ is an isomorphism over $U$ and the preimage of $U$ in $[V/G]$ has codimension at least 2 in $[V/G]$.
\end{proposition}
\begin{proof}
By \autoref{prop:generic-rep}, $V$ is a generic representation of $G$. The stack $\cX=[V/G]$ is tame with coarse space $\pi\colon\cX\to Y:=V/G$. Note that $\cX$ is equal to its own stable locus $\cX^s$. The proof of \cite[Proposition 10.3]{SatrianoUsatine4} starting from the phrase ``Thus, $\cX^s$ is tame'' applies word-for-word to show there is an open subset $U\subset V/G$ such that if $\cU=\pi^{-1}(U)$, then $\cX\setminus\cU$ has codimension at least $2$ in $\cX$, and $\cU\to U$ is an isomorphism. The proof of \cite[Proposition 10.4]{SatrianoUsatine4} then goes through to show that $Y\setminus U$ has codimension at least $2$ in $Y$ and 
\[
\cO_U \simeq \cO_\cU \simeq (\det L_\cX)|_\cU \simeq (\det L_\cX)|_U \simeq \omega_Y|_U,
\]
where the second equality uses \autoref{prop:LVmodG}; hence $\omega_Y$ is trivial.
\end{proof}

We may now prove \autoref{theoremQuotientIsQGorenstein}.

\begin{proof}[Proof of \autoref{theoremQuotientIsQGorenstein}]
To check that $V/G$ is Gorenstein, it is sufficient to check after base changing to the algebraic closure of $k$, so $V/G$ is Gorenstein by \autoref{prop:Q-Gorenstein} (Note that $V/G$ is Cohen-Macaulay by \cite[Main Theorem]{HochsterRoberts}). To check that $[V/G] \to V/G$ is birational and $K_{[V/G]/(V/G)} = 0$, it is sufficient to find an open subset $U \subset V/G$ such that $[V/G] \to V/G$ is an isomorphism over $U$ and the preimage of $U$ in $[V/G]$ has codimension at least 2 in $[V/G]$. Since this condition can be checked after base changing to the algebraic closure of $k$, we are done by \autoref{prop:Q-Gorenstein}.
\end{proof}

\section{Cyclotomic inertia of quotients by finite linearly reductive groups}

The main goal of this section is to show how \autoref{maintheoremMotivicMcKay} follows from \autoref{maintheoremGorensteinMeasureCrepantResolution} and a certain description for cyclotomic inertia of quotient stacks, which we will also prove in this section. We begin by setting some useful notation.

\begin{notation}
Let $r \in \Z_{>0}$, let $G$ be a finite linearly reductive group scheme over $k$, and let $x: B\mu_r \to BG$ be a representable morphism. If $k$ is algebraically closed, then there exists a group monomorphism $\widetilde{\phi}: \mu_r \to G$ such that $x$ is given (up to 2-isomorphism) by applying the pushout construction along $\widetilde{\phi}$ to $\mu_r$-torsors. In that case, we will let $\phi_x \in \Conj_{\mu_r}(G)$ denote the image of $\widetilde{\phi} \in \uHom_k^{\grp, \inj}(\mu_r, G)(k)$ in $\Conj_{\mu_r}(G)$. More generally, if $V$ is a scheme over $k$ with a $G$-action and $y: B\mu_r \to [V/G]$ is a representable morphism, we will set $\phi_y = \phi_x$, where $x$ is the composition of $y$ with $[V/G] \to BG$. Note that $\phi_y$ only depends on $y$ up to 2-isomorphism.
\end{notation}

\begin{notation}
If $G$ is a group scheme over $k$, $V$ is a scheme over $k$ with $G$-action, and $\widetilde{\phi}: \mu_r \to G$ is a group monomorphism, we will let $Z(\widetilde{\phi})$ denote the centralizer of $\widetilde{\phi}$ in $G$, and we will let $V^{\widetilde{\phi}}$ denote the $\widetilde{\phi}$-fixed locus of $V$. Since $V^{\widetilde{\phi}}$ is invariant under the restriction of the $G$-action to $Z_G(\widetilde{\phi})$, we will equip $V^{\widetilde{\phi}}$ with the resulting $Z_G(\widetilde{\phi})$-action.
\end{notation}

We now give a description for the cyclotomic inertia of global quotients by finite linearly reductive groups.

\begin{proposition}\label{propositionCyclotomicInertiaOfQuotient}
Assume $k$ is algebraically closed, let $G$ be a finite linearly reductive group scheme over $k$, and for every $r \in \Z_{>0}$ and $\phi \in \Conj_{\mu_r}(G)$, choose some group monomorphism $\widetilde{\phi}: \mu_r \to G$ whose image in $\Conj_{\mu_r}(G)$ is $\phi$. Let $V$ and $V'$ be smooth finite type schemes over $k$ with affine diagonal and $G$-action, and let $V \to V'$ be a $G$-equivariant morphism. Then there exist isomorphisms
\[
	I_{\mu}[V/G] \xrightarrow{\sim} \bigsqcup_{\phi \in \Conj_\mu(G)} [V^{\widetilde{\phi}} / Z_G(\widetilde{\phi})]
\]
and
\[
	I_{\mu}[V'/G] \xrightarrow{\sim} \bigsqcup_{\phi \in \Conj_\mu(G)} [(V')^{\widetilde{\phi}} / Z_G(\widetilde{\phi})]
\]
such that
\[
\xymatrix{
I_{\mu}[V/G]\ar[r]\ar[d] & \bigsqcup_{\phi \in \Conj_\mu(G)} [V^{\widetilde{\phi}} / Z_G(\widetilde{\phi})]\ar[d]\\
I_{\mu}[V'/G] \ar[r] & \bigsqcup_{\phi \in \Conj_\mu(G)} [(V')^{\widetilde{\phi}} / Z_G(\widetilde{\phi})]
}
\]
2-commutes and for every $k$-point $y$ of $I_{\mu_r}[V/G]$, the isomorphism $I_{\mu}[V/G] \xrightarrow{\sim} \bigsqcup_{\phi \in \Conj_\mu(G)} [V^{\widetilde{\phi}} / Z_G(\widetilde{\phi})]$ sends $y$ to the piece indexed by $\phi_y$.
\end{proposition}

\begin{proof}
Note that $I_{\mu_r}[V/G]$ is smooth by \autoref{propositionCyclotomicInertiaFiniteTypeAffineDiagonalSmooth} and \autoref{l:fppf-loc-sm->sm}. By \cite[Proposition 1.2]{Sala}, we have a natural equivalence
\[
I_{\mu_r}[V/G]\xrightarrow{\sim}[V_{\mu_r}/G],
\]
where $\uHom_k^{\grp, \inj}(\mu_r, G)$ is a finite scheme (see \cite[Proposition 1.1]{Sala}) parameterizing injective homomorphisms $\mu_r\to G$ and
\[
V_{\mu_r}\subset V\times\uHom_k^{\grp, \inj}(\mu_r, G)
\]
is the subscheme whose $S$-valued points are pairs $(v,\widetilde{\phi})$ with $\widetilde{\phi}\colon\mu_{\ell,S}\to G_S$, and $v\in V(S)$ a fixed point of $\widetilde{\phi}$; here $G$ acts on $V$ through its left action and acts on $\uHom_k^{\grp, \inj}(\mu_r, G)$ by conjugation.

For all injective maps $\widetilde{\phi}\colon\mu_r\to G$, we have an embedding
\[
V^{\widetilde{\phi}}\subset V_{\mu_r}
\]
given on $S$-valued points by $v\mapsto (v,\widetilde{\phi}_S)$. We see $V^{\widetilde{\phi}}$ is invariant under the centralizer $Z_G(\widetilde{\phi})\subset G$, which exists by \cite[Lemma 2.2.4]{ConradReductive}. As a result, we obtain a map
\[
[V^{\widetilde{\phi}}/Z_G(\widetilde{\phi})]\to[V_{\mu_r}/G]
\]
and hence a map
\[
\cY:=\coprod_{\Conj_{\mu_r}(G)}[V^{\widetilde{\phi}}/Z_G(\widetilde{\phi})]\xrightarrow{F}[V_{\mu_r}/G].
\]
Note that $V^{\widetilde{\phi}}$ is smooth and so $\cY$ is smooth by \autoref{l:fppf-loc-sm->sm}, hence $\cY$ is reduced; recall $\cX$ is also smooth as noted earlier. We prove $F$ induces an isomorphism $\cY\to\cX_{\red}=\cX$ by showing it is fully faithful and essentially surjective on the level of $S$-points for all reduced $k$-schemes $S$. Furthermore, since $\cY$ and $\cX$ are locally finitely presented over $k$, we may additionally assume $S$ is as well. Indeed, we may write $S=\lim_i S_i$ with the $S_i$ locally finitely presented and then $\cY(S) = \colim \cY(S_i)$; similarly for $\cX$.

We begin with the following observation. Given the object $(f,\eta)\in[V^{\widetilde{\phi}}/Z_G(\widetilde{\phi})](S)$ where $\eta\colon Z_G(\widetilde{\phi})\times S\to S$ is the trivial torsor and $f$ maps the identity section of $\eta$ to $v\in V^{\widetilde{\phi}}(S)$, then $F$ induces an isomorphism from the stabilizer of $(f,\eta)$ to that of $F(f,\eta)$. To see this, note that the stabilizer of $(f,\eta)$ is given by those elements of $Z_G(\widetilde{\phi})(S)$ which stabilizes $v$. On the other hand the stabilizer of $F(f,\eta)=(v,\widetilde{\phi}_S)$ is given by $\Stab(v)\cap Z_G(\widetilde{\phi})(S)$ and the map between the stabilizers is the identity map.

Next, to prove $F$ is faithful, we may consider an automorphism $\alpha$ of an object of $[V^{\widetilde{\phi}}/Z_G(\widetilde{\phi})](S)$ such that $F(\alpha)=\id$. Our object of $[V^{\widetilde{\phi}}/Z_G(\widetilde{\phi})](S)$ is given by a $Z_G(\widetilde{\phi})$-torsor $P\to S$ and $Z_G(\widetilde{\phi})$-equivariant map $P\to V^{\widetilde{\phi}}$. Since automorphisms of objects form a sheaf, to prove $\alpha=\id$, it suffices to look locally where $P$ is the trivial torsor; the result then follows from the trivial torsor case.

To prove $F$ is full, let $(f,\eta)\in[V^{\widetilde{\phi}}/Z_G(\widetilde{\phi})](S)$ and $(f',\eta')\in[V^{\widetilde{\phi}'}/Z_G(\widetilde{\phi})](S)$ and suppose we have an isomorphism $\alpha\colon F(f,\eta)\to F(f',\eta')$. To show $\alpha=F(\beta)$ for some $\beta$, we may look locally; indeed faithfulness ensures that any locally defined $\alpha$ descends. Thus, after looking locally, we may assume $\eta$ and $\eta'$ are the trivial torsor and $S$ is connected. Thus, $(f,\eta)$ and $(f',\eta')$ are given by $v\in V^{\widetilde{\phi}}(S)$ and $v'\in V^{\widetilde{\phi}'}(S)$, and $\alpha$ corresponds to $g\in G(S)$ such that $gv=v'$ and $g\widetilde{\phi}_S g^{-1}=\widetilde{\phi}'_S$. Since $S$ is reduced, $G(S)=G_{\red}(S)$ and $G_\red$ is a finite constant \'etale group scheme (see e.g., the first paragraph of the proof of \autoref{l:descending-mur->well-split} below); thus, since $S$ is connected, $g=h_S$ for some $h\in G(k)$. We therefore have two $k$-valued points $h\widetilde{\phi} h^{-1}$ and $\widetilde{\phi}'$ of the scheme $\uHom_k^{\grp, \inj}(\mu_r, G)$ which become equal over $S$; since $S$ is locally finitely presented (and flat) over $k$, $S\to\Spec k$ is an fppf cover, hence $h\widetilde{\phi} h^{-1}=\widetilde{\phi}'$. By our hypothesis that the disjoint union in the definition of $\cY$ runs over $k$-points of $\uHom_k^{\grp, \inj}(\mu_r, G) / G$, we see $\widetilde{\phi}=\widetilde{\phi}'$. As a result, $g\in Z_G(\widetilde{\phi})$ and $gv=v'$, hence defines our desired $\alpha$.

Lastly, to show $F$ is essentially surjective, full faithfulness along with the fact that $\cX$ is a stack ensure that we may look locally; as a result, we may start with a connected $S$ and an $S$-point of $V_{\mu_r}$, i.e., a pair $(v,\psi)$ with $\psi\in\uHom_k^{\grp, \inj}(\mu_r, G)(S)$ and $v\in V^\psi(S)$. Since $\uHom_k^{\grp, \inj}(\mu_r, G)$ is finite and $S$ is reduced and connected, our $S$-point factors through an $L$-point of $\uHom_k^{\grp, \inj}(\mu_r, G)$ with $L/k$ finite. %By \autoref{l:descending-mur->well-split}, we have $\psi=\phi_S$ for some injective $\phi\colon\mu_r\to G$.
Since $k$ is algebraically closed, $L=k$. Then simply observe that the $S$-point of $[V^{\widetilde{\phi}}/Z_G(\widetilde{\phi})]$ corresponding to the trivial torsor and $v\in V^{\widetilde{\phi}}(S)$ maps under $F$ to $(v,\psi)$.

This shows our desired equivalence
\[
I_\mu[V/G]\xrightarrow{\sim}\bigsqcup_{\phi}[V^{\widetilde{\phi}}/Z_G(\widetilde{\phi})].
\]
We see that if $y$ is a $k$-point of $I_\mu[V/G]$ corresponding to $B\mu_r\to [V/G]$, then precomposing with $\Spec k\to B\mu_r$ and choosing a lift to $V$, we obtain a $k$-point $v\in V(k)$. Under the equivalence from \cite[Proposition 1.2]{Sala}, we see $y$ maps to the point of $[V_{\mu_r}/G]$ induced by $(v,\phi_y)\in V_{\mu_r}(k)$. Then under our equivalence $F$, we see this point is mapped to the component indexed by $\phi_y$.

Finally, a $G$-equivariant morphism $V\to V'$ induces $[V^{\widetilde{\phi}}/Z_G(\widetilde{\phi})]\to [(V')^{\widetilde{\phi}}/Z_G(\widetilde{\phi})]$ for all $\phi$. We see directly from the definition of $F$ that
\[
\xymatrix{
\bigsqcup_{\Conj_{\mu_r}(G)}[V^{\widetilde{\phi}}/Z_G(\widetilde{\phi})]\ar[r]\ar[d] & [V_{\mu_r}/G]\ar[d]\\
\bigsqcup_{\Conj_{\mu_r}(G)}[(V')^{\widetilde{\phi}}/Z_G(\widetilde{\phi})]\ar[r] & [V'_{\mu_r}/G]
}
\]
commutes. Combined with naturality of the equivalence in \cite[Proposition 1.2]{Sala}, this completes the proof.
\end{proof}

\begin{example}\label{ex:singularHominj}
We show that although $I_{\mu_r}(BG)$ is smooth, $\uHom_k^{\grp, \inj}(\mu_r, G)$ may be singular. Consider $G=\mu_p\rtimes\ZZ/2$ over $K$ with $\cha(K)=p\neq2$. Then $\mu_2=\ZZ/2$, so $\uHom_k^{\grp, \inj}(\mu_2, G)$ is equal to the $2$-torsion $G[2]$ minus the identity section. This is the non-identity component of $G$, which, as a scheme, is $\mu_p$. This is non-reduced yet $I_{\mu_2}(G)$ is still smooth. Explicitly, this is because $\mu_p\subset G$ acts via the squaring map which is simple and transitive, so $[\uHom_k^{\grp, \inj}(\mu_2, G)/G]=[(\mu_p/\mu_p)/\ZZ/2]=B(\ZZ/2)$.
\end{example}

We may now set the following notation.

\begin{notation}
Assume $k$ is algebraically closed, let $G$ be a finite linearly reductive group scheme over $k$, and let $V$ be a smooth finite type scheme over $k$ with affine diagonal and $G$-action. If $\cY$ is a connected component of $I_\mu[V/G]$, then we will set $\phi_{\cY} = \phi_y$, where $y$ is some $k$-point of $\cY$. By \autoref{propositionCyclotomicInertiaOfQuotient}, this notation does not depend on the choice of $y$.
\end{notation}

We will now prove the following two corollaries of \autoref{propositionCyclotomicInertiaOfQuotient}.

\begin{corollary}\label{corollaryClassOfInertiaComponent}
Assume $k$ is algebraically closed, let $V$ be a finite dimensional vector space over $k$, and let $G$ be a finite linearly reductive subgroup scheme of $\GL(V)$. If $\cY$ is a connected component of $I_\mu[V/G]$, then
\[
	\e(\cY) = \bL^{\dim \cY} \e(BZ_G(\phi_{\cY})).
\]
\end{corollary}

\begin{proof}
Since $k$ is algebraically closed, there exists some $r \in \Z_{>0}$ and group monomorphism $\widetilde{\phi}: \mu_r \to G$ such that $\phi_\cY$ is the image of $\widetilde{\phi}$ in $\Conj_\mu(G)$. Since $G$ acts linearly on $V$, the fixed locus $V^{\widetilde{\phi}}$ is a linear subspace of $V$ and in particular is connected. Thus by \autoref{propositionCyclotomicInertiaOfQuotient},
\[
	\cY \cong [V^{\widetilde{\phi}} / Z_G(\widetilde{\phi})].
\]
Since $V^{\widetilde{\phi}}$ is a vector space over $k$ and $Z_G(\widetilde{\phi})$ acts on it linearly, $[V^{\widetilde{\phi}} / Z_G(\widetilde{\phi})]$ is the total space of a vector bundle on $BZ_G(\widetilde{\phi})$, so
\[
	\e(\cY) = \e([V^{\widetilde{\phi}} / Z_G(\widetilde{\phi})]) = \bL^{d} \e(BZ_G(\widetilde{\phi})) = \bL^{d} \e(BZ_G(\phi_{\cY}))
\]
for some $d \in \Z$. Furthermore $d = \dim\cY$ by \autoref{propositionDimensionAndGrothendieckNorm}, so we are done.
\end{proof}

\begin{corollary}\label{corollaryComponentsAndConjugacyClasses}
Assume $k$ is algebraically closed, let $V$ be a finite dimensional vector space over $k$, and let $G$ be a finite linearly reductive subgroup scheme of $\GL(V)$. Then the assignment $\cY \mapsto \phi_\cY$ gives a bijection from the set of connected components of $I_\mu[V/G]$ to the set $\Conj_\mu(G)$.
\end{corollary}

\begin{proof}
This is a straightforward consequence of \autoref{propositionCyclotomicInertiaOfQuotient} and the fact that, since $G$ acts linearly on $V$, each of the $[V^{\widetilde{\phi}} / Z_G(\widetilde{\phi})]$ from the statement of \autoref{propositionCyclotomicInertiaOfQuotient} is nonempty and connected.
\end{proof}

\begin{proposition}\label{propositionShiftOfInertiaComponentAge}
Assume $k$ is algebraically closed, let $V$ be a finite dimensional vector space over $k$, and let $G$ be a finite linearly reductive subgroup scheme of $\GL(V)$. If $\cY$ is a connected component of $I_\mu[V/G]$, then
\[
	\shft_{[V/G]}(\cY) + \dim\cY = \age(\phi_\cY).
\]
\end{proposition}

\begin{proof}
Since $k$ is algebraically closed, there exists some $r \in \Z_{>0}$ and group monomorphism $\widetilde{\phi}: \mu_r \to G$ such that $\phi_\cY$ is the image of $\widetilde{\phi}$ in $\Conj_\mu(G)$. Let $x: B\mu_r \to BG$ be the map induced by $\widetilde{\phi}$.

Since $G$ acts linearly on $V$, the inclusion of the origin into $V$ is a $G$-equivariant section of the map $V \to \Spec(k)$, so we have an induced map $f: BG \to [V/G]$ whose composition with $g: [V/G] \to BG$ is 2-isomorphic to the identity map on $BG$. Therefore, setting $y = f \circ x$, we have $\phi_y = \phi_x = \phi_\cY$. Thus by \autoref{corollaryComponentsAndConjugacyClasses}, $y$ is a $k$-point of $\cY$. Therefore
\begin{align*}
	\shft_{[V/G]}&(\cY)+ \dim\cY = \\
	&(1/r)\sum_{w = 1}^r w[ \dim_{k} H^0((Ly^* L_{[V/G]}(-w)) - \dim_{k}H^1((Ly^*L_{[V/G]})(-w))].
\end{align*}
We also have the exact triangle
\[
	Lg^* L_{BG} \to L_{[V/G]} \to L_{[V/G]/BG}.
\]
Pulling back along $f$ and noting that $g \circ f \cong \id_{BG}$, we get an exact triangle
\[
	L_{BG} \to Lf^* L_{[V/G]} \to Lf^* L_{[V/G]/BG}.
\]
Recall that $f$ is given by taking the $G$-quotient of the inclusion of the origin 0 into $V$. Therefore $Lf^* L_{[V/G]/BG}$, thought of as a $G$-equivariant vector space over $k$, is (quasi-isomorphic to) $\Omega^1_{V,0}$ with the induced $G$-action. In other words, $Lf^* L_{[V/G]/BG}$ is (quasi-isomorphic to) $V$ with its $G$-action. Thus we have an exact triangle
\[
	L_{BG} \to Lf^* L_{[V/G]} \to V.
\]
Pulling back along $x$, we get an exact triangle
\[
	Lx^*L_{BG} \to Ly^* L_{[V/G]} \to V,
\]
where $V$ has the $\mu_r$-action given by restricting its $G$-action along $\widetilde{\phi}$. We therefore get the following exact sequence of $\mu_r$-equivariant vector spaces:
\begin{align*}
	0 \to &\cH^0(Lx^*L_{BG}) \to \cH^0( Ly^* L_{[V/G]}) \to V \to\\
	&\cH^1(Lx^*L_{BG}) \to \cH^1(Ly^* L_{[V/G]}) \to 0.
\end{align*}
Since the coarse space map $B\mu_r \to \Spec(k)$ is cohomologically affine, we have that for all $w \in \Z$,
\begin{align*}
	0 \to &H^0((Lx^*L_{BG})(-w)) \to H^0(( Ly^* L_{[V/G]})(-w)) \to V(-w)^{\mu_r} \to\\
	& H^1((Lx^*L_{BG})(-w)) \to H^1((Ly^* L_{[V/G]})(-w)) \to 0
\end{align*}
is an exact sequence of vector spaces over $k$, so
\begin{align*}
	&\dim_k H^0(( Ly^* L_{[V/G]})(-w)) - \dim_k H^1((Ly^* L_{[V/G]})(-w)) \\
	&= \dim_k V(-w)^{\mu_r} -\left[ \dim_k H^1((Lx^*L_{BG})(-w)) - \dim_k H^0((Lx^*L_{BG})(-w))  \right].
\end{align*}
We already showed that if we take the left hand side of the above equality, multiply by $w/r$, and sum over $w = 1, \dots, r$, we get $\shft_{[V/G]}(\cY) + \dim\cY$. By the definitions of $\age$ and $\overline{\wt}$, if we do the same thing to the right hand side, we get $\age(\phi_\cY) - \overline{\wt}_{BG}(x)$. Therefore
\[
	\shft_{[V/G]}(\cY) + \dim\cY = \age(\phi_\cY) - \overline{\wt}_{BG}(x).
\]
The desired result then follows from \autoref{prop:LBG}.
\end{proof}

\begin{corollary}\label{corollaryAgeIsInteger}
Assume $k$ is algebraically closed, let $V$ be a finite dimensional vector space over $k$, and let $G$ be a finite linearly reductive subgroup scheme of $\SL(V)$. If $\phi \in \Conj_\mu(G)$, then $\age(\phi) \in \Z$.
\end{corollary}

\begin{proof}
By \autoref{corollaryComponentsAndConjugacyClasses} and \autoref{propositionShiftOfInertiaComponentAge}, it is sufficient to show $\shft_{[V/G]}(\cY) \in \Z$ for every connected component $\cY$ of $I_\mu[V/G]$, but this follows from \autoref{maintheoremGorensteinMeasureCrepantResolution} and \autoref{theoremQuotientIsQGorenstein}.
\end{proof}

We may now prove \autoref{maintheoremMotivicMcKay}.

\begin{proof}[Proof of \autoref{maintheoremMotivicMcKay}]
By \autoref{theoremQuotientIsQGorenstein} and \autoref{maintheoremGorensteinMeasureCrepantResolution}, we have that $V/G$ is Gorenstein, $\bL^{\ord_{V/G}^\Gor}$ is integrable on $\sL(V/G)$, and
\[
	\e_{\str}(V/G) = \sum_{\cY} \bL^{\shft_{[V/G]}(\cY)}\e(\cY),
\]
where the sum varies over all connected components $\cY$ of $I_\mu([V/G])$. Thus if $k$ is algebraically closed, \autoref{corollaryClassOfInertiaComponent} and \autoref{propositionShiftOfInertiaComponentAge} give that
\[
	\e_{\str}(V/G) = \sum_{\cY} \bL^{\age(\phi_\cY)} \e(BZ_G(\phi_\cY)),
\]
and the desired result thus follows from \autoref{corollaryComponentsAndConjugacyClasses}.
\end{proof}

We note that $k$ being algebraically closed is a stronger than necessary hypothesis for the conclusion of \autoref{maintheoremMotivicMcKay} to hold. The key reason we assumed $k$ is algebraically closed is to guarantee that every point of $\uHom_k^{\grp, \inj}(\mu_r, G)$ is defined over $k$. In case a version of \autoref{maintheoremMotivicMcKay} with weaker hypotheses becomes of use, we prove the following lemma, which shows that every point of $\uHom_k^{\grp, \inj}(\mu_r, G)$ is defined over $k$ under the hypothesis that $k$ has enough roots of unity and $G$ is \emph{well-split} in the sense of \cite[Definition 2.9]{AOV}

\begin{lemma}\label{l:descending-mur->well-split}
Let $k$ be a field with $\cha(k)=p$, let $r=p^en$ with $\gcd(p,n)=1$, and suppose $k$ contains all $n$-th roots of unity. Let $G=\Delta\rtimes H$ where $\Delta$ is a connected finite diagonalizable group scheme and $H$ is a finite \'etale tame constant group. 

If $K/k$ is a finite extension of fields, then every homomorphism $\phi\colon\mu_{r,K}\to G_K$ is the base change of a homomorphism $\mu_{r,k}\to G$.
\end{lemma}
\begin{proof}
Let $j\colon H=G_{\red}\to G$ denote the canonical inclusion. The splitting $s\colon H\to G$ is a section of a seperated map, hence a closed immersion. Since $H$ is reduced, it must factor through $G_{\red}$, and thus yields an isomorphism of schemes $H\simeq G_{\red}$. This therefore endows $G_{\red}$ with the structure of a subgroup scheme of $G$.

We have $\mu_r=\mu_{p^e}\times\bZ/n$. Let $\iota_1\colon\mu_{p^e,K}\to\mu_{r,K}$ and $\iota_2\colon(\bZ/n)_K\to\mu_{r,K}$ denote the inclusions. Then $\phi\iota_1$ must factor through $\Delta_K$. Since both $\mu_{p^e,K}$ and $\Delta_K$ are diagonalizable, there exists $\psi_1\colon\mu_{p^e,k}\to\Delta$ such that $\psi_{1,K}=\phi\iota_1$. Since $(\bZ/n)_K$ is reduced $\phi\iota_2$ must factor through $G_{\red,K}=H_K$. Note that $(\bZ/n)_K\to H_K$ is \emph{a priori} a map of schemes but it is, in fact, a group scheme map since $(\bZ/n)_K\to H_K\subset G_K$ is. Since $\bZ/n$ and $H$ are constant groups, there is a group homomorphism $\psi'_2\colon\bZ/n\to H$ and we see $\psi_2:=j\psi'_2$ satisfies $\psi_{2,K}=\psi_2$.

Now define $\psi$ to be the composition $m\circ(\psi_1\times\psi_2)$ where $m\colon G\times G\to G$ is group multiplication. We see $\psi_K=\phi$, so it remains to prove that $\psi$ is a group homomorphism. This follows by fppf descent since $\phi$ is group homomorphism.
\end{proof}

We end this section with an example that illustrates some of the subtleties that can arise in the absence of the hypotheses of \autoref{l:descending-mur->well-split}.

\begin{example}\label{ex:quadratic-twist-mu3}
We illustrate that $I_{\mu_r}(BG)$ may have coarse space whose residue fields are non-trivial extensions of the ground field. Let $L/K$ be a separable quadratic extension with $\bZ/2=\Gal(L/K)$ acting on $\mu_{3,L}$ via inversion. Let $G$ be the associated twisted form over $K$. Since $G$ is abelian, $I_{\mu_3}(BG) = \uHom^{\grp, \inj}(\mu_3, G) \times BG$. Note that $\uHom^{\grp, \inj}(\mu_3, G)=\uIsom^\grp(\mu_3, G)$. We see $\uIsom^\grp(\mu_3, G)_L=\Spec L\coprod\Spec L$ where the two points correspond to the identity map and the inverse map. The Galois action swaps these two points, so $\uIsom^\grp(\mu_3, G)=\Spec L$, and 
\[
I_{\mu_3}(BG) = \Spec L \times_K BG
\]
whose coarse space has residue field $L$.
\end{example}

\section{Components of cyclotomic inertia and irreducible representations}

Our next goal is to prove \autoref{maincorollaryEulerMcKay} by proving the following.

\begin{proposition}\label{prop:Euler-char-irreps}
If $G$ is a finite linearly reductive group scheme over an algebraically closed field $k$, then
\[
\#\Conj_\mu(G)=\#\Irrep_k(G).
\]
\end{proposition}

We require some preliminary results.

\begin{lemma}\label{l:p-group-normalized-form}
Let $p$ be a prime, $Q$ be a finite abelian $p$-group, $H$ a finite group with order prime to $p$, and let $\alpha\colon H\to\Aut(Q)$ define a semi-direct product $G'=Q\rtimes H$. Then every element of $G'$ is $Q$-conjugate to an element of the form $(q,h)$ with $\alpha(h)q=q$. Furthermore, if $(q,h)$ and $(q',h')$ satisfy $\alpha(h)q=q$ and $\alpha(h')q'=q'$, then $(q,h)$ and $(q',h')$ are $G'$-conjugate if and only if they are $H$-conjugate.
\end{lemma}
\begin{proof}
We write the group structure on $Q$ additively and on $H$ multiplicatively. To ease notation, we write $\alpha(h)q$ simply as $hq$. Then $(a,g)(q,h)(a,g)^{-1}=(gq+(1-ghg^{-1})a,ghg^{-1})$. Let $m$ be the order of $h$, which is necessarily prime to $p$. Taking
\[
a = -\frac{1}{m} \sum_{i=1}^{m-1} ih^iq.
\]
and $g=1$, we see the above expression is $(\pi_h(q),h)$ where $\pi_h=\frac{1}{m}\sum_{i=0}^{m-1}h^i$. Note that $\pi_h(q)$ is fixed by $h$.

Next, let $(q,h)$ and $(q',h')$ satisfy $hq=q$ and $h'q'=q'$. Suppose
\[
(q',h')=(a,g)(q,h)(a,g)^{-1}=(gq+(1-ghg^{-1})a,ghg^{-1}).
\]
Note that $q'$ and $gq$ are both fixed by $h'=ghg^{-1}$, so $(1-ghg^{-1})a=(1-h')a$ must also be fixed by $h'$. Thus, $(1-h')a=\pi_h(1-h')a=0$. So, we see
\[
(q',h')=(gq,ghg^{-1})=(0,g)(q,h)(0,g)^{-1},
\]
i.e., $(q,h)$ and $(q',h')$ are in the same $H$-conjugacy class.
\end{proof}

\begin{lemma}\label{l:Clifford-theory-indep-k}
Let $A$ be a finite abelian group and $H$ a finite group. Let $\alpha\colon H\to\Aut(A)$ be a group map defining a group scheme $G=D_\bZ(A)\rtimes H$ over $\bZ$, where $D_S(A):=\uHom(A,\bG_{m,S})$ is the Cartier dual. Then, for every field $k$, a $G_k$-representation $V$ is irreducible if and only if it is of the form
\[
V=\bigoplus_{a\in Hb}V_a
\]
with $V_b$ an irreducible $H_b$-representation, and $V_a\subset V$ the subspace where $\Delta_k$ acts with weight $a$.

In particular, for every $k$ with $\cha(k)$ prime to $|H|$, the quantity
\[
\#\Irrep_k(G_k)
\]
is independent of $k$.
\end{lemma}
\begin{proof}
Let $\Delta=D_\bZ(A)$ and $k$ be a field. The characterization of irreducible $G_k$-representations is nearly the same as in \cite[Proposition 25]{Serre}. % and let $k(-b)$ denote the $1$-dimensional representation of $\Delta_k$ where $\Delta_k$ acts as weight $b$.

Let $V$ be a $G_k$-representation. We first characterize when $V$ is irreducible. Restricting $V$ to a $\Delta_k$-representation, we obtain a decomposition $V=\bigoplus_{a\in A}V_a$. Let $\Supp(V)$ denote the $a\in A$ for which $V_a\neq0$. For every $h\in H$, we have $h(V_a)=V_{ha}$ where $ha:=\alpha(h)a$; indeed, for any $R$-valued points $\delta\in\Delta(R)$ and $v\in V_a(R)$, we see $\delta hv=h(h^{-1}\delta h)v=a(h^{-1}\delta h)\cdot hv=(ha)(\delta)\cdot hv$; this shows $h(V_a)\subset V_{ha}$ but the action by $h$ is an automorphism of $V$. Next, for any $b\in A$, we see $\bigoplus_{a\in Hb}V_a$ is a $G$-subrepresentation of $V$, hence if $V$ is irreducible, $\Supp(V)$ must be a single $H$-orbit. Furthermore, we see that $V_b$ is %of the form $k(-b)\otimes W$, where $W$ is
a representation of $H_b:=\{h \in H\mid hb=b\}$. If $V$ is irreducible then $V_b$ must be irreducible; indeed, if $U\subset V_b$ were a non-trivial $H_b$-subrepresentation, then $\bigoplus_{h\in H/H_a}hU\subset V$ would be a non-trivial $G$-subrepresentation. Conversely, if $V$ is reducible and $\Supp(V)$ is a single $H$-orbit, then letting $V'\subset V$ be a non-trivial $G$-subrepresentation, we see $\Supp(V')=\Supp(V)$ and $V'=\bigoplus_{a\in Hb}(V'\cap V_b)$; indeed, some $V'_a$ is non-zero then acting by $H$ shows that all $V'_a\neq0$ for $a\in Hb$. Now if $V_b$ were an irreducible $H_b$-representation, then $V'\cap V_b=V_b$. However, again acting by $H$, we see $V_{hb}=h(V'\cap V_b)=hV'\cap V_{hb}$, so $V'=V$, a contradiction.

We have therefore shown that $V$ is irreducible if and only if it is of the form $V=\bigoplus_{a\in Hb}V_a$ with $V_b$ an irreducible $H_b$-representation. As a result,
\[
\#\Irrep_k(G_k)=\sum_{b\in A/H}\#\Irrep_k(H_b).
\]
Note that $A/H$ and $H_b$ depend only on the $H$-action on $A$, and not on the field $k$. If $\cha(k)$ is prime to $|H|$ (hence $|H_b|$ is as well), then $\#\Irrep_k(H_b)=\#\Conj(H_b)$ is also independent of $k$.
\end{proof}

\begin{proof}[{Proof of \autoref{prop:Euler-char-irreps}}]
When $\cha(k) = 0$, this is a standard fact, noting that $\Conj_\mu(G)$ is in bijection with the number of conjugacy classes in $G$ in this case. Thus we will assume that $\cha(k)=p$. By \cite[Lemma 2.11 and Proposition 2.13]{AOV}, we may write $G=\Delta\rtimes H$ with $\Delta$ the connected component of the the identity and $H$ the component group, and furthermore, $\Delta$ is a finite diagonalizable group scheme and $H$ is a tame constant group. Let $A$ be the Cartier dual of $\Delta$. Then let $\Delta^\C$ be the abelian $p$-group $\Hom(A,\C^*)=\Hom(A,\bQ/\bZ)$, and set $G^\C = \Delta^\C \rtimes H$, where the semi-direct product is given by the action of $H$ on $\Delta^\C$ induced in the obvious way by the action of $H$ on $\Delta$. By \autoref{l:p-group-normalized-form}, we see $\Conj(G^\C)$ is in bijection with the set
\[
\{(\phi,h)\in \Delta^\C \times H\mid h\phi=\phi\}/H
\]
where $H$ acts as $g(\phi,h)=(g\phi,ghg^{-1})$.

Next, let $r=p^em$ with $m$ prime to $p$, so that $\mu_{r,k}=\mu_{p^e,k}\times\mu_{m,k}$. As in the proof of \autoref{l:descending-mur->well-split}, giving a morphism $\varphi\colon\mu_{r,k}\to G$ is equivalent to giving morphisms $\varphi_1\colon\mu_{p^e,k}\to\Delta$ and $\varphi_2\colon\mu_{m,k}\to G_{\red}=H\subset G$ such that the images commute. Since $k$ is algebraically closed, we may fix a choice of isomorphism $\mu_{m,k}\simeq\bZ/m$; then giving $\varphi_2$ is equivalent to giving an order $m$ element of $H$. Giving $\varphi_1$ is equivalent to giving a surjective map $A\to\bZ/p^e$, i.e., an order $p^e$ element of $\Hom(A,\bQ/\bZ)=\Delta^\C$; note that since $A$ is a $p$-group, all elements of $\Delta^\C$ have order equal to a power of $p$. Thus, as we vary over $r$ (equivalently, vary over $e$ and $m$), we see
\[
\sum_r \#\uHom_k^{\grp, \inj}(\mu_r,G)(k)=\#\{(\phi,h)\in \Delta^\C \times H\mid h\phi=\phi\},
\]
where the condition $h\phi=\phi$ is equivalent to commutativity of $\varphi_1$ and $\varphi_2$. Since $k$ is algebraically closed,
\[
(\uHom_k^{\grp, \inj}(\mu_r,G)/G)(k)=\uHom_k^{\grp, \inj}(\mu_r,G)(k)/G(k)=\uHom_k^{\grp, \inj}(\mu_r,G)(k)/H.
\]
Thus,
\[
\#\Conj_\mu(G)=\sum_r \#(\uHom_k^{\grp, \inj}(\mu_r,G)/G)(k)=\#\{(\phi,h)\in \Delta^\C \times H\mid h\phi=\phi\}/H,
\]
where the $H$-action is given by $g(\phi,h)=(g\phi,ghg^{-1})$. We therefore find 
\[
	\# \Conj_\mu(G)=\#\Conj(G^\C)=\#\Irrep_{\C}(G^{\C}),
\] 
where the second equality follows from the characteristic 0 case. By \autoref{l:Clifford-theory-indep-k}, $\#\Irrep_\C(G^\C)=\#\Irrep_k(G)$, and we are done.
\end{proof}

We end this section by completing the proof of \autoref{maincorollaryEulerMcKay}.

\begin{proof}[Proof of \autoref{maincorollaryEulerMcKay}]
The theorem is immediate from \autoref{maintheoremMotivicMcKay} and \autoref{prop:Euler-char-irreps}.
\end{proof}

\section{Purity of crepant resolutions}

Recall that if $Y$ is a finite type scheme over $k$ and $k^s$ is a separable closure of $k$, then there exists a \emph{weight filtration} on each $H^i_{\mathrm{\acute{e}t,c}}(Y \otimes_k k^s, \Q_\ell)$ as follows. If $Y'$ is a finite type scheme over a finitely generated subfield $k'$ of $k$ such that $Y \cong Y' \otimes_{k'} k$, then the filtration on $H^i_{\mathrm{\acute{e}t,c}}(Y \otimes_k k^s, \Q_\ell)$ given by the canonical weight filtration on $H^i_{\mathrm{\acute{e}t,c}}(Y' \otimes_{k'} k'^s, \Q_\ell)$ is independent of the choice of $Y'$ and $k'$. We will say $H^i_{\mathrm{\acute{e}t,c}}(Y \otimes_k k^s, \Q_\ell)$ is \emph{pure of weight $i$} if for all $j \neq i$, we have $\gr^W_jH^i_{\mathrm{\acute{e}t,c}}(Y \otimes_k k^s, \Q_\ell)=0$. When $k$ is finitely generated, this coincides with the usual definition of purity. The main goal of this section is to prove the following.

\begin{theorem}\label{thm:crepant-res-VmodG->pure}
Let $k^s$ be a separable closure of $k$, and let $\ell$ be a prime number that is invertible in $k$. Let $V$ be a finite dimensional vector space over $k$, let $G$ be a finite linearly reductive subgroup scheme of $\SL(V)$, and let $Y \to V/G$ be a crepant resolution of singularities by a scheme $Y$. Then for all $i \in \Z_{\geq 0}$, we have $H_{\mathrm{\acute{e}t,c}}^i(Y \otimes_k k^s, \Q_\ell)$ is pure of weight $i$.
\end{theorem}

\begin{remark}
In the special case where $k$ has characteristic 0, \autoref{thm:crepant-res-VmodG->pure} was proved in \cite[Theorem 8.4]{BatyrevNonArch}. Most of the work below is in proving \autoref{l:lift-alg-gp-actions-crepant-res}, which is a generalization to positive characteristic of \cite[Proposition 8.2]{BatyrevNonArch}. We note that we were unable to adapt the argument for \cite[Proposition 8.2]{BatyrevNonArch} in loc. cit. to positive characteristic, so our argument for \autoref{l:lift-alg-gp-actions-crepant-res} is quite different from the argument in loc. cit.
\end{remark}

Before we prove \autoref{thm:crepant-res-VmodG->pure}, we use it to prove \autoref{maintheoremCohomologicalMcKay}.

\begin{proof}[Proof of \autoref{maintheoremCohomologicalMcKay}]
By replacing $k$ with a finitely generated subfield over which $G$ and $X \to V/G$ are defined, we may assume that $k$ is a finitely generated field. Then by \autoref{thm:crepant-res-VmodG->pure},
\[
	\EP(X) = \sum_{i \in \Z_{\geq 0}} (-1)^i \dim_{\Q_\ell} H_{\mathrm{\acute{e}t,c}}^i(X \otimes_k k^s, \Q_\ell)  t^i,
\]
where $\EP(X)$ is the Euler-Poincar\'{e} polynomial of $X$. On the other hand, if $\overline{k}$ is an algebraic closure of $k$, then
\[
	\EP(X) = \EP(X \otimes_k \overline{k}) = \EP(\e(X \otimes_k \overline{k})) = \EP(\e_{\str}(V/G \otimes_k \overline{k})),
\]
where the first and second equalities are by, e.g., \cite[Chapter 2 Proposition 3.5.10]{ChambertLoirNicaiseSebag}). Noting that \cite[Proposition 3.1(iii)]{Ekedahl2} implies that $\EP(\e(BH)) = 1$ for any finite group scheme $H$, \autoref{maintheoremMotivicMcKay} gives that
\[
	\EP(\e_{\str}(V/G \otimes_k \overline{k})) = \sum_{\phi \in \Conj_{\mu}(G \otimes_k \overline{k})} t^{2 \age(\phi)}, 
\]
and we are done.
\end{proof}

The remainder of this section will now be used to prove \autoref{thm:crepant-res-VmodG->pure}.

\subsection{Extending actions of connected algebraic groups}

The following proposition shows that, under suitable conditions, actions of connected algebraic groups can be imported from other varieties which differ in codimension at least $2$.

\begin{proposition}\label{prop:T-action-extends-codim2}
Let $G$ be a connected smooth affine group scheme over an algebraically closed field $k$. Suppose $G$ acts on a reduced Noetherian affine scheme $X=\Spec A$ over $k$. Consider the following commutative diagram
\[
\xymatrix{
W\ar[r]^-{\jmath'}\ar[d]_-{\jmath} & U'\ar[d]^-{\pi'}\\
Y\ar[r]^-{\pi} & X
}
\]
with $\pi$ is projective and $\pi'$ is separated, $\jmath$ and $\jmath'$ are open immersions, $\codim(Y\setminus W)\geq2$, and $\codim(U'\setminus W)\geq2$. Assume $Y$ and $U'$ are smooth over $k$, and that $U'$ has a $G$-action making $\pi'$ equivariant. (Note that $W\subset U'$ is not assumed to be $G$-invariant.) Then there is a unique $G$-action on $Y$ such that the map $\pi$ is equivariant.

\end{proposition}
\begin{proof}
Let $\cL$ be a very ample line bundle on $Y$ such that $Y=\Proj\bigoplus_{n\geq0}H^0(Y,\cL^{\otimes n})$. Due to the codimension conditions,
\[
H^0(Y,\cL^{\otimes n})=H^0(W,\jmath^*\cL^{\otimes n})=H^0(U',\jmath'_*\jmath^*(\cL^{\otimes n})).
\]
Note that $\cL'_n:=\jmath'_*\jmath^*\cL^{\otimes n}$ is a line bundle by \cite[Proposition 1.9]{HartshorneReflexive} and $(\jmath')^*\cL'_n=(\jmath')^*((\cL')^{\otimes n})$ by \cite[Proposition 1.6]{HartshorneReflexive}, so $\cL'_n=(\cL')^{\otimes n}$. As a result,
\[
H^0(Y,\cL^{\otimes n})=H^0(U',(\cL')^{\otimes n}).
\]
Next, \cite[Theorem 1.6]{Sumihiro2} says that after replacing $\cL'$ (and hence $\cL$) by a higher tensor power, it has a $G$-linearization. In particular, this endows $H^0(Y,\cL^{\otimes n})$, and hence $Y$, with a $G$-action.

To see that $\pi$ is $G$-equivariant, since $U'\to X=\Spec A$ is equivariant, $A\to H^0(\cO_{U'})\to \bigoplus_n H^0(U',(\cL')^{\otimes n})$ are equivariant maps. Hence, $A\to \bigoplus_n H^0(Y,\cL^{\otimes n})$ is equivariant, and so $\pi$ is.

Lastly, we show the $G$-action on $Y$ is unique. Any $G$-action on $Y$ making $\pi$ equivariant forces the $G$-action on $Y$ to agree with the $G$-action on $X$ restricted to $X^{\sm}$. Furthermore, two choices of maps $G\times Y\to Y$ that agree on $G\times X^{\sm}$ must be equal because $G\times X^{\sm}$ is reduced and $Y$ is separated.
\end{proof}

The same method of proof used in \autoref{prop:T-action-extends-codim2} also shows \autoref{prop:T-action-extends-codim2-not-used} below. We do not use this next result but we state it since we were unable to find it in the literature.

\begin{proposition}\label{prop:T-action-extends-codim2-not-used}
Let $G$ be a connected smooth affine group scheme over an algebraically closed field $k$. Suppose $G$ acts on a reduced Noetherian affine scheme $X=\Spec A$ over $k$. Let $\pi\colon Y\to X$ be a projective morphism and assume $Y$ is normal and reduced. Let $\jmath\colon U\to Y$ be a dominant open immersion with $\codim(Y\setminus U)\geq2$, and let $G$ act on $U$ such that $\pi\jmath$ is equivariant. Then there is a unique $G$-action on $Y$ such that the maps $\pi$ and $\jmath$ are equivariant.
%Let $S$ be a scheme and $G\to S$ a surjective smooth affine group scheme with connected fibers. Let $Y\to X$ be a projective morphism of Noetherian schemes over $S$ and assume $Y$ is normal. Let $U\subset Y$ be a non-empty open dense subset with $\codim(Y\setminus U)\geq2$. Then every $G$-action on $U$ extends uniquely to a $G$-action on $Y$.
\end{proposition}
\begin{proof}
Let $\cL$ be a very ample line bundle on $Y$ such that $Y=\Proj\bigoplus_{n\geq0}H^0(Y,\cL^{\otimes n})$. Then \cite[Theorem 1.6]{Sumihiro2} says that after replacing $\cL$ by a higher tensor power, we may assume there is a $G$-linearization on $\jmath^*\cL$. Thus, there is a $G$-action on $H^0(\jmath^*\cL^{\otimes n})=H^0(\cL^{\otimes n})$ for all $n>0$. This yields a $G$-action on $Y$ making $\jmath$ equivarant. To see that $\pi$ is $G$-equivariant, note that the two arrows in the diagram
\[
\xymatrix{
G\times Y\ar[d]_-{\id\times\pi}\ar[r] & Y\ar[d]^-{\pi}\\
G\times X\ar[r] & X
}
\]
agree on $G\times U$ as $U\to X$ is $G$-equivariant. Since $G\times U$ is reduced and $X$ is separated, we see the above diagram commutes. Lastly, to show the $G$-action on $Y$ is unique, any two choices of maps $G\times Y\to Y$ that agree on $G\times U$ must be equal because $G\times U$ is reduced and $Y$ is separated. 
\end{proof}

\subsection{Purity}

We prove the following result which says that in large generality, actions by connected smooth algebraic groups lift to crepant resolutions.

\begin{proposition}\label{l:lift-alg-gp-actions-crepant-res}
Let $G$ be a connected smooth geometrically irreducible affine group scheme over an algebraically closed field $k$. Let $X$ be a finitely presented scheme over $k$ with finite linearly reductive quotient singularities, and let $\pi\colon Y\to X$ be a crepant resolution of singularities. If $G$ acts on $X$, then $Y$ admits a $G$-action making $\pi$ equivariant.
\end{proposition}

We prove this result after showing a technical lemma.

\begin{lemma}\label{l:codim2-setup}
Let $G$ be a connected smooth geometrically irreducible affine group scheme over a field $k$. Let $X$ be a finitely presented scheme over $k$ and let $\pi\colon Y\to X$ and $\pi'\colon Y'\to X$ be strong resolutions of singularities. Suppose further $\pi$ is crepant and that there is a $G$-action on $Y'$ making $\pi'$ equivariant. Then there is a $G$-invariant open subset $U'\subset Y'$ and a commutative diagram
\[
\xymatrix{
W\ar[r]^-{\jmath'}\ar[d]_-{\jmath} & U'\ar[d]^-{\pi'|_{U'}}\\
Y\ar[r]^-{\pi} & X
}
\]
with $\jmath$ and $\jmath'$ open immersions, $\codim(Y\setminus W)\geq2$, and $\codim(U'\setminus W)\geq2$.
\end{lemma}
\begin{proof}
Since $\pi'$ is $G$-equivariant and $X^{\sm}$ is $G$-invariant, it follows that the exceptional locus $Y'\setminus X^{\sm}$ is also $G$-invariant. If $Z'$ is an irreducible component of the exceptional locus then $G\times Z'$ is irreducible by \cite[Tag 038F]{stacks-project} and its image under the action map $G\times Z'\subset G\times Y'\to Y'$ is irreducible and contains $Z'$, hence equals $Z'$; as a result, every irreducible component of $Y'\setminus X^{\sm}$ is $G$-invariant.

Let $U'\subset Y'$ be the open subset whose complement is the union of all $\pi'$-discrepant divisors and all components of $Y'\setminus X^{\sm}$ with codimension at least $2$. Let $\eta'_1,\dots,\eta'_m$ be the generic points of all $\pi'$-crepant divisors $D'_1,\dots,D'_m$. Under the identification $k(U') = k(Y)$, each $\eta'_i$ yields a valuation $\eta_i$ on $k(Y)$ and the map $k(Y) \to k(U')$ restricts to an isomorphism $\cO_{Y, \eta_i} \xrightarrow{\simeq} \cO_{U', \eta'_i}$. It follows from, e.g., \cite[Proposition B.1]{RydhApproximation}, that there is an open subset $U'_i\subset U'$ containing $\eta'_i$ where the rational map $U'\dasharrow Y$ is regular. Let $V'=\bigcup_i U'_i$ and $f\colon V'\to Y$ be the induced regular map such that $\pi'|_V=\pi f$. %reduced + separated argument shows the maps agree on overlaps. is this necessary to say?

Since $f$ induces an isomorphism $\cO_{Y, \eta_i} \xrightarrow{\simeq} \cO_{U', \eta'_i}$, the equivalence of properties (1) and (6) from \cite[Tag 02GU]{stacks-project} shows that $f$ is \'etale at all $\eta'_i$. Thus, there is an open subset $W\subset V'$ containing all $\eta'_i$ such that $f|_W$ is \'etale, hence quasi-finite. Note that $f|_W$ is separated since $\pi$ and $\pi'|_W$ are. Since $f|_W$ is birational, by Zariski's Main Theorem,
%https://mathoverflow.net/questions/78696/is-there-an-intuitive-reason-for-zariskis-main-theorem
$f|_W$ is an open immersion.

It remains to prove $\codim(U'\setminus W)\geq2$ and $\codim(Y\setminus W)\geq2$. This follows from the general result that if $\eta$ is a crepant valuation on $X$ (i.e., a divisorial valuation of $k(X)$ with vanishing discrepancy), then for any resolution $Z\to X$, there exists a divisor $D\subset Z$ with generic point $z$ such that $\cO_{Z,z}\subset k(X)$ is the DVR associated to $\eta$. Indeed, \cite[Lemma 2.45]{KollarMori} says there exists some proper birational map $Z'\to Z$ from a normal variety $Z'$ with this property. Then \cite[Corollary 2.31]{KollarMori} implies that the discrepancy of $\eta$ is at least $1$, contradicting the fact that $\eta$ is a crepant valuation.
\end{proof}

\begin{proof}[{Proof of \autoref{l:lift-alg-gp-actions-crepant-res}}]
We begin by constructing an auxiliary resolution $\pi'\colon Y'\to X$ which is $G$-equivariant. Let $Y'\to X$ be the functorial resolution of \cite[Theorem E]{BerghRydh}. Since this resolution commutes with smooth base change, the pullback $(G\times X)\times_X Y'$ under either the projection or action map $G\times X\to X$ is given by $G\times Y'$. Thus, by descent, the $G$-action on $X$ lifts to a $G$-action on $Y'$ making $\pi'$ equivariant. Then combining \autoref{prop:T-action-extends-codim2} and \autoref{l:codim2-setup}, we obtain a $G$-action on $Y$ making $\pi\colon Y\to X$ equivariant.
\end{proof}

The remainder of the proof of \autoref{thm:crepant-res-VmodG->pure} essentially follows the proof of \cite[Theorem 8.4]{BatyrevNonArch}.

\begin{proof}[{Proof of \autoref{thm:crepant-res-VmodG->pure}}]
By our definition of purity, we reduce immediately to the case where $k$ is algebraically closed. Consider the diagonally embedded $\bG_m\subset\GL(V)$ which acts on $V$ and commutes with the action of $\bG_m$. Thus, the $\bG_m$-action descends to $X:=V/G$. By \autoref{l:lift-alg-gp-actions-crepant-res}, we obtain a $\bG_m$-action on $Y$ making $\pi\colon Y\to X$ equivariant. By \cite[Theorem 7.8.14 (2)--(4)]{AlperNotes}, the fixed locus $Y^{\bG_m}$ is smooth with connected components $F_i$, there is a filterable stratification $Y=\coprod_i Y_i$ with each $Y_i\subset Y$ locally closed (note we are using that since $Y$ is proper over $V/G$ and the so-called \emph{attractor locus} $(V/G)^+$ surjects onto $V/G$, we have that the attractor locus $Y^+$ surjects onto $Y$), and there are Zariski local affine fibrations $Y_i\to F_i$. Since $V^{\bG_m}=\{0\}$, we see $X^{\bG_m}$ is a point; since $Y^{\bG_m}\to X^{\bG_m}$ is proper, it follows that each $F_i$ is smooth and proper, hence $H_{\mathrm{\acute{e}t,c}}^j(F_i \otimes_k k^s, \Q_\ell)$ is pure of weight $j$. As a result, $H_{\mathrm{\acute{e}t,c}}^j(Y_i \otimes_k k^s, \Q_\ell)$ is also pure of weight $j$. Using that our stratification is filterable, an induction argument as in \cite[Lemma 8.3]{BatyrevNonArch} shows that $H_{\mathrm{\acute{e}t,c}}^j(Y \otimes_k k^s, \Q_\ell)$ is pure of weight $j$.
\end{proof}

\section{Example}

We end this paper with an explicit example.

\begin{example}

Let $k$ be an algebraically closed field of characteristic $3$, and let $G=\mu_3\rtimes\bZ/2$ where $\bZ/2$ acts on $\mu_3$ by inversion. Let $V=\bA^3$ and let $G\subset\SL(V)$ given by the following action. We let $\mu_3$ act on $V$ with weights $(1,-1,0)$, and let the non-trivial element $\tau\in\bZ/2$ act by
\[
\tau(x,y,z)=(y,x,-z).
\]
Then
\[
V/G=(\Spec k[x,y,z]^{\mu_3})^{\bZ/2} = \Spec k[x^3,y^3,xy,z]^{\bZ/2}.
\]
To calculate this invariant ring, note that $\tau$ acts by swapping $x^3$ and $y^3$, fixing $xy$, and negating $z$. Thus,
\[
V/G=\Spec k[x^3+y^3,xy,(x^3-y^3)z,z^2,(x^3-y^3)^2]
\]
Note that $(x^3-y^3)^2$ is expressible in terms of the other invariants. Thus, letting
\[
u=x^3+y^3,\quad v=z^2,\quad t=xy,\quad w=(x^3-y^3)z,
\]
we see
\[
V/G=\Spec k[u,v,w,t]/(w^2-v(u^2-t^3)).
\]

We show that $V/G$ has a crepant resolution $X\to V/G$ through explicit blow-ups. We see $V/G\subset\bA^4$ is a hypersurface. We first blow up along the singular line $L=V(u,t,w)\simeq\bA^1$; note that our hypersurface vanishes to order $2$ at $L$, and $L$ is codimension $2$ in $V/G$, hence, this first blow-up is crepant. One checks that the singular locus of the blow-up is an $\bA^1$-family of $A_1$-singularities.
%An explicit calculation shows that the $w$-chart is smooth, and the $t$- and $u$-charts have a common singular locus given by the product of an $A_1$-singularity with $\bA^1$.
Blowing up the singular locus is therefore another crepant map. Furthermore, from this description, we see
\[
\e(X)=\e(V/G)-\bL^1+(\bL+1)\bL - \bL + (\bL+1)\bL = \e(V/G)+2\bL^2.
\]
We claim that $\e(V/G)=\bL^3$. To see this, we project our hypersurface $V/G\subset\bA^4$ to $\bA^2_{u,t}$. Over the cuspidal curve $u^2=t^3$, we see $w=0$ and $v$ is is free. On the other hand, on the complement of the cuspidal curve, $w$ is free and $v$ is determined. Thus, $\e(V/G)=(\bL^2 - \bL)\bL + \bL^2 = \bL^3$. It follows that
\[
\e_{\str}(V/G)=\bL^3+2\bL^2.
\]

We now illustrate our main theorems for this example. For \autoref{maintheoremMotivicMcKay}, we first note that the only morphisms $\mu_r\to G$ are given by $r=1,2,3$. For $r=3$, there are two maps $\mu_3\to\mu_3\subset G$ given by $\zeta\mapsto\zeta^\pm$, but these are conjugate under the action of $\tau$; it is straightforward to calculate that the centralizer of this map is $\mu_3$. Since each $V_w$ is $1$-dimensional, the age of this map is given by $\frac{1}{3}(1+2+3)=2$. For $r=2$, up to conjugation, there is only the unique map $\mu_2\xrightarrow{\simeq}\bZ/2\subset G$. We see $\dim V_1=2$ and $\dim V_2=1$, so the age of this map is $\frac{1}{2}(2+2)=2$. Hence, \autoref{maintheoremMotivicMcKay} states
\[
\e_{\str}(V/G) =\bL^3 \e(BG)+\bL^2 \e(B\mu_3)+\bL^2 \e(B\bZ/2).
\]
We claim $\e(BG)=1$. To see this, consider the $2$-dimensional vector bundle $[\bA^2/G]\to BG$ where $\mu_3\subset G$ acts on $\bA^2$ with weights $(1,-1)$, and $\tau\in G$ acts by $\tau(x,y)=(y,x)$. Then decomposing we see
\[
\bL^2\e(BG)=\e([\bA^2/G])=\e(BG)+\e([((\bG_m\times0)\cup(0\times\bG_m))/G]+\e([\bG_m^2/G]).
\]
The middle term is given by $[\bG_m/\mu_3]=\bG_m$, so its class is $\bL-1$. To compute the last term, we see $k[x^\pm,y^\pm]^G=k[X^\pm,Y^\pm]^{\Z/2}$, where $\tau$ fixes $Y=xy$ and maps $X=x^{3}$ to $Y^3/X$. Note that there are $\Z/2$-stabilizers exactly on the curve $X^2-Y^3$ in $\bG_m^2=\bG_m^2/\mu_3$, and this curve has class $\bL-1$. Away from this curve, $\Z/2$ acts freely. We see $k[X^\pm,Y^\pm]^{\Z/2}=k[Y^\pm,(X^2-Y^3)/X]$. Putting these observations together, we have
\[
\e([\bG_m^2/G])=(\bL-1)+(\bL-1)^2,
\]
and hence
\[
(\bL^2-1)\e(BG)=(\bL-1)+(\bL-1)+(\bL-1)^2=\bL^2-1,
\]
proving $\e(BG)=1$. Therefore, \autoref{maintheoremMotivicMcKay} tells us
\[
\e_{\str}(V/G) =\bL^3+2\bL^2
\]
which agrees with our direct computation of $\e_{\str}(V/G)$ coming from the crepant resolution.

\autoref{maincorollaryEulerMcKay} %combined with \autoref{l:Clifford-theory-indep-k}
tells us
\[
\chi_{\str}(V/G)=\#\Irrep(G)=\#\Conj_\mu(G)=3.
\]

Lastly, \autoref{maintheoremCohomologicalMcKay} tells us the dimensions of the \'etale cohomology groups of our crepant resolution $X\to V/G$. We see
\[
\dim_{\bQ_\ell} H^4_{\mathrm{\acute{e}t,c}}(X,\bQ_\ell)=2,\quad\dim_{\bQ_\ell} H^6_{\mathrm{\acute{e}t,c}}(X,\bQ_\ell)=1,
\]
and all other cohomology groups vanish.

\end{example}

\bibliographystyle{alpha}
\bibliography{MCLRFGSPC}

\end{document}